\documentclass[12pt]{article}
\usepackage[utf8]{inputenc}
\usepackage{geometry}
\usepackage[colorlinks]{hyperref}

\usepackage{graphicx}
\usepackage{tabularx}
\usepackage{color}

\usepackage[bb=boondox]{mathalfa}  

\usepackage{amssymb,amsmath,amsthm}
\usepackage{mathtools}
\usepackage{mathrsfs}  
\usepackage{fancyhdr}
\usepackage{enumitem}
\usepackage{bbm}
\usepackage{dsfont}
\usepackage{tikz-cd}
\usepackage{pgfplots}
\usepackage{yhmath}
\usepackage{ragged2e}
\usepackage{float}
\usepackage{extarrows}
\usepackage{yhmath}
\newcommand{\N}{\mathbb{N}}
\newcommand{\Z}{\mathbb{Z}}
\newcommand{\R}{\mathbb{R}}

\newcommand{\E}{\mathbb{E}}

\newcommand{\F}{\mathcal{F}}

\newcommand{\X}{\mathcal{X}}
\newcommand{\Prob}{\mathbb{P}}

\numberwithin{equation}{section}

\newtheorem{theorem}{Theorem}[section] 
\newtheorem{corollary}[theorem]{Corollary}
\newtheorem{proposition}[theorem]{Proposition}
\newtheorem{lemma}[theorem]{Lemma}

\theoremstyle{definition}
\newtheorem{definition}[theorem]{Definition}

\theoremstyle{remark}
\newtheorem{remark}[theorem]{Remark}

\makeatletter
\def\thanks#1{\protected@xdef\@thanks{\@thanks
        \protect\footnotetext{#1}}}
\makeatother

\title{On the stationary measures of two variants of the voter model}

\date{}
\author{Jhon Astoquillca\thanks{Bernoulli Institute, University of Groningen. } \thanks{Groningen Cognitive System and Materials Center, University of Groningen} \thanks{j.k.astoquillca.aguilar@rug.nl} \thanks{https://orcid.org/0000-0002-8913-3867} }

\begin{document}

\maketitle

\begin{abstract}
In the voter model, vertices of a graph (interpreted as voters) adopt one out of two opinions (0 and 1), and update their opinions at random times by copying the opinion of a neighbor chosen uniformly at random. This process is dual to a system of coalescing random walks. The duality implies that the set of stationary measures of the voter model on a graph is linked to the dynamics of the collision of random walks on this graph. By exploring the key ideas behind this relationship, we characterize the sets of stationary measures for two variations of the voter model: first, a version that incorporates interchanging of opinions among voters, and second, the voter model on dynamical percolation. To achieve these results, we analyze the collision properties of random walks in two contexts: first, with a swapping behavior that complicates collisions, and second, with random walks defined on a dynamical percolation environment.
\end{abstract}

{
  
}


\section{Introduction}

The voter model on the locally finite graph $G=(V,E)$ is a Markov process, denoted by~$(\eta_t)_{t \geq 0}$, with the configuration space $\{0,1\}^V$ and stochastic dynamics described informally as follows: each vertex~$x$ of~$V$ is provided with an independent exponential clock of rate 1. Each time the clock of~$x \in V$ rings, it updates its current state by copying the state of~$y$ that is chosen uniformly at random among all vertices neighboring $x$. 

This interacting particle system can be taken as a model for the collective behavior of voters who constantly update their opinions. It was introduced by Clifford and Sudbury in \cite{CliffordSudbury73} and Holley and Liggett in \cite{HolleyLiggett75}, and is treated in the expository texts \cite{Liggett2005}, \cite{Swart2016}.

Given $\alpha \in [0,1]$, let $\pi^\mathrm{site}_\alpha$ denote the Bernoulli($\alpha$) product measure on~$\{0,1\}^{V}$. It is well-established in the literature of the voter model (see aforementioned expository texts) that for any $\alpha \in [0,1]$, starting the process at time 0 from $\pi^\mathrm{site}_\alpha$, its distribution at time $t$ converges weakly, as $t \to \infty$, to a limiting probability measure $\mu_\alpha$. This gives rise to a one-parameter family of (distinct) distributions~$\{\mu_\alpha: \alpha \in [0,1]\}$. Note that $\mu_0$ and $\mu_1$ correspond to Dirac masses on the identically 0 and identically~1 configurations, respectively; we call these the consensus measures. Moreover, 
\begin{itemize}
    \item[\textbf{(A)}] if two independent continuous-time random walks on $G$  meet almost surely, then~$\mu_0$ and~$\mu_1$ are the only extremal stationary measures of~$(\eta_t)$ (that is, the only stationary measures that cannot be expressed as non-trivial convex combinations of other stationary measures), and
    $$ \forall \alpha, \hspace{0.5mm} \mu_\alpha = \alpha \mu_1 + (1-\alpha)\mu_0. $$
    \item[\textbf{(B)}] If two independent continuous-time random walks on $G$  meet finitely many times almost surely, and the tail $\sigma$-algebra of a simple random walk on $G$ is trivial, then $\mu_\alpha$ is extremal for any $\alpha$, and in fact, the set of extremal measures of the voter model is equal to the family $\{\mu_\alpha: \alpha \in [0,1]\}$. 
\end{itemize}

Although these statements, with this level of generality, cannot be found in the existing literature to the best of our knowledge, their proofs are straightforward extensions of those of the corresponding results on~$\Z^d$. In Section \ref{MetaTheoremSection}, we briefly present the proofs, as they serve as the cornerstone for establishing the identical characterizations of two variations of the voter model.

Let us briefly describe two consequences of these results. First, on the Euclidean lattice~$(\Z^d,E(\Z^d))$, for $d=1$ or 2, the only extremal stationary measures of the voter model are $\mu_0$ and $\mu_1$. For~$d \geq 3$, the set of extremal stationary measures is $\{\mu_\alpha$, $\alpha \in [0,1]\}.$  

Second, we consider the Bernoulli bond percolation on the Euclidean lattices (see \cite{Grimmett99}). Let~$p_c (\in (0,1) \text{ for } d \geq 2 )$ be the critical parameter of the model; for any~$d \geq 2$ and~$p > p_c$ there exists a unique infinite cluster~$\mathcal{C}_\infty$ a.s. If~$d \ge 3$, this cluster almost surely satisfies the two properties mentioned in point~\textbf{(B)} above: two independent continuous-time simple random walks on $\mathcal{C}_\infty$ meet finitely many times (Theorem~1.2 of \cite{Chen2010}) and the tail $\sigma$-algebra of a continuous-time random walk on~$\mathcal{C}_\infty$ is trivial (Theorem~4b of \cite{Barlow2004}). Therefore, the set of extremal stationary measures for the voter model on $\mathcal{C}_\infty$ is the family~$\{\mu_\alpha$, $\alpha \in [0,1]\}$. 

The first variant of the classical voter model investigated in this article arises by adding another dynamic to the voters: they can interchange their opinions. Given $\mathsf{v} \in (0,\infty)$ and~$d \ge 2$ integer, the voter model with stirring on a~$d$-regular graph $G $ is a Markov process, denoted by~$(\xi_t)_{t \geq 0}$, with the configuration space $\{0,1\}^V$ and stochastic dynamics described informally as follows: each site $x \in V$ is provided with an exponential clock of rate~1, and each edge~$e$ of~$E$ is provided with an exponential clock of rate $\mathsf{v}$ (all independently). Each time the clock of $x$ rings, it updates its current state by copying the state of $y$ that is chosen uniformly at random among its neighbors. Each time the clock of~$e = \{x,y\}$ rings, the current states of~$x$ and~$y$ are interchanged. In Section~\ref{IVMSection}, we give the formal definition of the model.

\begin{theorem}\label{ThmCharacterizationIVM}
If the voter model with stirring on a $d$-regular graph~$G$~($d \ge 2$) with parameter~$\mathsf v \in (0,\infty)$ is started from $\pi^\mathrm{site}_\alpha$, $\alpha \in [0,1]$, then its distribution at time~$t$ converges weakly as~$t \to \infty$ to a measure~$\mu^\mathrm{stir}_\alpha$. This family of (distinct) measures is stationary for~$(\xi_t)$. Moreover,
\begin{enumerate}
    \item if two independent continuous-time simple random walks on $G$ meet almost surely, then the only extremal stationary measures for~$(\xi_t)$ are~$\mu^\mathrm{stir}_0$ and~$\mu^\mathrm{stir}_1$: the consensus measures; 
    \item if two independent continuous-time simple random walks on $G$ meet finitely many times almost surely, and the tail $\sigma$-algebra of a simple random walk on $G$ is trivial, then the set of extremal stationary measures for $(\xi_t)$ is~$\{\mu^\mathrm{stir}_\alpha, \alpha \in [0,1]\}$. 
\end{enumerate}
\end{theorem}
When $G = \Z^d$ we also have the following result:
\begin{proposition}\label{prop_intro_ergo}
    For any $d \ge 1$ and~$\alpha \in [0,1]$, the measure~$\mu^\mathrm{stir}_\alpha$ is ergodic with respect to translations of~$\Z^d$.
\end{proposition}

The second variant of the classical voter model investigated in this article arises by adding simple time dynamics to the ordinary percolation model. Dynamical percolation was introduced by Häggström, Peres and Steif in \cite{PeresSteif1997}. For a survey of the model see \cite{Steif2009}. Let us briefly describe the model; a more formal definition will be given in Section \ref{VMonDPSection}. Given the parameters $p \in [0,1]$ and $\mathsf{v} \in (0,+\infty)$, dynamical percolation on any graph $G=(V,E)$ is a Markov process on~$\{0,1\}^E$, denoted by $(\zeta_t)_{t \geq 0}$. Its stochastic dynamics is described as follows: every edge evolves independently as a two-state Markov chain, where an edge in state 0 (closed) flips to state 1 (open) at rate $p \cdot \mathsf{v}$ and an edge in state 1 flips to state 0 at rate $(1-p) \cdot \mathsf{v}$. The unique stationary distribution of this model is $\pi^\mathrm{edge}_p$, the Bernoulli($p$) product measure on $\{0,1\}^E$. 

Given $R \in \N$, we define the voter model with range $R$ on dynamical percolation on the Euclidean lattice, denoted by 
$$\{ M_t\}_{t \geq 0} = \{(\eta_t,\zeta_t)\}_{t \geq 0},$$
by running $(\zeta_t)$ as a dynamical percolation and providing each vertex $x$ of $\Z^d$ with an independent exponential clock of rate 1. Each time the clock of $x$ rings, it chooses~$y$ uniformly at random in~$B_1(x,R)$, the~$\ell^1$-metric ball of center~$x$ and radius~$R$, and \emph{attempts} to copy the opinion of~$y$. The attempt is only successful if there is a path of open edges in $\zeta_t$ inside~$B_1(x,R)$ connecting $x$ and $y$.
In Section \ref{VMonDPSection} we give a more formal definition of the model.

The following statement has two parts, the first of which was already proved in~\cite{Hutchcroft2022}.
\begin{theorem}\label{ThmCharacterization}
If the voter model on dynamical percolation on the Euclidean lattice with parameters~$\mathsf{v} \in (0,\infty)$,~$p \in [0,1]$ and $R \in \N$ is started from $\pi^\mathrm{site}_\alpha \otimes \pi^\mathrm{edge}_p$, $\alpha \in [0,1]$, then its distribution at time $t$ converges weakly, as $t \to \infty$, to a measure $\mu^{\mathrm{dyn}}_\alpha$.
The family of (distinct) measures $\{\mu^{\mathrm{dyn}}_\alpha: \alpha \in [0,1]\}$ is stationary for $(M_t)$ and ergodic with respect to translations of~$\Z^d \times E(\Z^d)$. Moreover, 
\begin{enumerate}
    \item \cite{Hutchcroft2022} for $d =1$ or 2, the only extremal stationary measures for $(M_t)$ are $\mu^{\mathrm{dyn}}_0$ and $\mu^{\mathrm{dyn}}_1$: the product of the consensus measures and $\pi^\mathrm{edge}_p$; 
    \item for $d\geq 3$, the set of extremal stationary measures for $(M_t)$ is $\{\mu^{\mathrm{dyn}}_\alpha, \alpha \in [0,1]\}$. 
\end{enumerate}
\end{theorem}

\subsection{Related work}
We have already mentioned the references~\cite{Chen2010} and~\cite{Hutchcroft2022}; the former studies the voter model and coalescing random walks on (static) supercritical percolation clusters, and the latter studies those same processes in dynamic percolation. Let us now mention a few additional references.

In \cite{ShanChen12}, the authors assign independent and identically distributed weights larger than 1 to the edges of~$\Z^2$, and use these weights to bias the choice of neighbors for opinion updates. They prove that for almost all choices of environment, the only extremal stationary measures of the voter model are the consensus measures. 

Stationary measures of the contact process on random dynamic environments (dynamical percolation among them) have been studied in \cite{Seiler2021}. The author uses the reversibility of the environment to study the survival probability of the infection process.

\subsection{Organization}
In Section \ref{background}, we give a characterization result of the stationary measures of the voter model on general graphs: we prove \textbf{(A)} and \textbf{(B)}.

Section \ref{IVMSection} introduces the interchange voter model on the Euclidean lattice. It includes its duality properties and the characterization of its stationary measures. We prove Theorem~\ref{ThmCharacterizationIVM} in Section~\ref{sec_charac_stat_stir} and Proposition~\ref{prop_intro_ergo} in Section~\ref{sec_ergo_stat_stir}.

We begin the study of the voter model on dynamical percolation in Section \ref{VMonDPSection}. We study collisions of random walks on dynamical percolation in Section~\ref{RandomWalksonDPSection}. We define a dual process and define a family of stationary measures in Section~\ref{duality_stat_measures_vmodyn_section}. We proof Theorem~\ref{ThmCharacterization} in two sections. First, we give the characterization in Section~\ref{section_characterization_stat_measures} and show the ergodicity of the measures $\mu^\mathrm{dyn}_\alpha$ in Section~\ref{section_ergodicty_dyn_perc}.

\section{Background on the voter model}\label{background}
We write $\N = \{1,2,\dots\}$ and $\N_0 = \N \cup \{0\}$. Given a set $A$, we denote by $|A|$ its cardinality and by $\mathcal{P}_{\text{fin}}(A)$ the set of all finite subsets of~$A$.

\subsection{Voter model on graphs}\label{VoterModelonGSection}
In this section, we state the characterization of the stationary measures of the voter model on connected locally finite graphs. The following steps to prove this are classical, see Section~2 of \cite{Liggett2004}, but we include them here for two reasons. First, we explain them in higher generality than the usual setting of~$\mathbb Z^d$. Second, they lay the groundwork for extensions to yet more general settings.

Given a locally finite graph~$G = (V,E)$, we denote the space of site configurations by~$\Omega = \{0,1\}^V$. For any sites~$x,y \in V$, we write~$x \sim y$ when~$\{x,y\} \in E$. We denote by $\deg(x)$ the number of edges that are incident to the vertex~$x$. The identically 0 and identically~1 configurations, denoted~$\bar{0}$ and~$\bar{1}$ respectively, are said to be the consensus configurations, and~$\delta_{ \bar{0} }$ and~$\delta_{ \bar{1} }$ the consensus measures.

The voter model $(\eta_t)_{t \geq 0}$ on $G$ is a Feller process with Markov pre-generator
$$ \mathcal{L}f(\eta) = \sum_{ \substack{x,y \in V \\  x \sim y } } \frac{f(\eta^{y \to x})- f(\eta)}{\mathrm{deg}(x)}, $$
where $f:\Omega \to \R$ is any function that only depends on finitely many coordinates,~$\eta \in \Omega$ and 
 
\begin{equation}
    \label{eq_notation_conf}
\eta^{y \to x}(z) = \left\{ \begin{array}{ll}
    \eta(z) & \text{if } z \neq x  \\
    \eta(y) & \text{if } z = x.
\end{array} \right.    
\end{equation}

Given $\eta \in \Omega$, we denote by $P_{\eta}$ a probability measure under which $(\eta_t)_{t \geq 0}$ is defined and satisfies $P_{\eta}\big( \eta_0 = \eta \big) =1$. Likewise, given a probability distribution $\mu$ on $\Omega$, we write~$P_\mu = \int P_\eta \; \mu(\mathrm{d} \eta)$. 

Given $k \in \N$, we will consider a system of coalescing random walks~$(Y^1_t, \ldots, Y^k_t)_{t \ge 0}$, defined as the continuous-time Markov chain on~$V^k$ with generator
\begin{equation}\label{Gen_crw}
    \mathcal{L}_\mathrm{crw}f(y_1,\dots,y_k) = \sum_{x \in V} \frac{1}{\deg(x)} \sum_{ y \sim x } [f(y^{x \to y}_1, \dots, y^{x \to y}_k) - f(y_1,\dots,y_k)],
\end{equation}
where for $x,y,z \in V$,
\begin{equation}\label{z_x_arrow_y}
     z^{x \to y} := \left\{ \begin{array}{cc}
        y & \text{if } z = x,  \\
        z & \text{otherwise.}
    \end{array} \right.    
\end{equation}
We write~$A_t := \{Y_t^1,\ldots, Y_t^k\}$ for $t \ge 0$. Given $A \in \mathcal{P}_\mathrm{fin}(V)$, we denote by~$P_A$ a probability measure under which $(A_t)_{t \ge 0}$ is defined and satisfies $P_A(A_0 = A) = 1$.

It is known that such a system of coalescing random walks is dual to the voter model, in the sense that
\begin{equation}\label{DualityMetaTheorem}
    P_\eta( \eta_t \equiv 1 \text{ on } A) = P_A( \eta \equiv 1 \text{ on } A_t)  \quad \text{for any } A \in \mathcal P_{\mathrm{fin}}(V).  
\end{equation}    
For~$\alpha \in [0,1]$, let~$\pi^\mathrm{site}_\alpha$ be the Bernoulli($\alpha$) product measure on~$V$. Using~\eqref{DualityMetaTheorem}, we obtain that
\begin{equation}\label{eq_dual_Tt}
P_{\pi^\mathrm{site}_\alpha}(\eta_t \equiv 1 \text{ on } A) = E_A[\alpha^{|A_t|}] \quad \text{for any } A \in \mathcal P_{\mathrm{fin}}(V).
\end{equation}

Since~$|A_t|$ is non-increasing in~$t$, we can define $|A_\infty| = \lim_{t \to \infty}|A_t|$. By taking the limit as $t \to \infty$ in~\eqref{eq_dual_Tt}, we see that $\eta_t$ under $P_{\pi^\mathrm{site}_\alpha}$ converges in distribution to a measure $\mu_\alpha$ on~$\{0,1\}^V$ satisfying 
\begin{equation}\label{MualphaMetaTheorem}
    \mu_\alpha( \eta: \eta \equiv 1 \text{ on } A ) = E_A \big[ \alpha^{|A_\infty|} \big] \quad \text{for any } A \in \mathcal P_{\mathrm{fin}}(V).
\end{equation}
As the measures $\mu_\alpha$ are obtained as distributional limits of $\eta_t$, they are also stationary with respect to the dynamics of the voter model \cite[Proposition~1.8.d]{Liggett2005}.

\subsection{Characterization of stationary measures}\label{MetaTheoremSection}
We state the characterization of the stationary measures of the voter model on a graph $G$ based on the collision properties of random walks on $G$. 


The references \cite{PeresKrishnapur2004, BarlowPeresSousi2012} are pioneering studies of the so-called finite (and infinite) collision property of graphs, which are notions pertaining to the number of encounters between two independent simple random walks on the graph. These references consider discrete-time walks, whereas here we will need the corresponding notions for continuous-time walks:
\begin{definition}
Let~$G$ be a connected graph, and let~$(X^1_t)_{t \geq 0}$ and $(X^2_t)_{t \geq 0}$ be independent continuous-time simple random walks on~$G$. We say that~$G$ has the \textbf{continuous infinite collision property} if 
\begin{equation}\label{eq_infinite_collision}
    P(\forall s > 0 \; \exists t \ge s: X^1_t = X^2_t) = 1
\end{equation}
for any starting positions of the two walks.
We say that~$G$ has the \textbf{continuous finite collision property} if 
\begin{equation}\label{eq_finite_collision}
    P(\exists s >0:\; X^1_t \neq X^2_t \;\; \forall t \ge s) = 1
\end{equation}
for any starting positions of the two walks.
\end{definition}



In \cite{Chen2008}, it is proved that on quasi-transitive graphs with subexponential growth, the continuous finite collision property is equivalent to the (discrete) finite collision property.


Let~$(W_t)_{t \geq 0}$ be a continuous-time Markov chain on a countable state space~$S$. We say that this process has trivial tail~$\sigma$-algebra~$\mathcal{T} = \cap_{t \geq 0} \sigma( \{ W_s : s \geq t \} )$ if for each~$i~\in~S$,~$P_i(A)~\in~\{0,1\}$ for any $A \in \mathcal{T}$, where~$P_i$ is the measure under which~$W_0 = i$.

We let $\mathcal{I}$ and~$\mathcal{I}_e$ denote the sets of stationary measures and the set of extremal stationary measures of the voter model on~$G$, respectively. 

With all of this notation we rewrite \textbf{(A)} and \textbf{(B)} of the introduction:
\begin{theorem}\label{MetaTheorem}
Let $(\eta_t)_{t \geq 0}$ be a voter model on a locally finite and connected graph $G$.  
\begin{enumerate}
    \item[\emph{\textbf{(A)}}] If $G$ has the continuous infinite collision property, then $\mathcal{I}_e = \{ \delta_{\overline{0}}, \delta_{\overline{1}} \}$.
    \item[\emph{\textbf{(B)}}] If $G$ has the continuous finite collision property and the tail $\sigma$-algebra $\mathcal{T}$ of a simple random walk on $G$ is trivial, then $\mathcal{I}_e = \{ \mu_\alpha: \alpha \in [0,1] \}$.
\end{enumerate}
\end{theorem}
We note the following points:
\begin{enumerate}
    \item If $G$ is transient and has bounded degrees, then $G$ has the continuous finite collision property. Indeed, under these assumptions, in \cite{PeresKrishnapur2004}, it is proved that the expected number of meetings between two independent walkers is finite. Following the same method we can prove the same for continuous-time random walks.
    \item If a continuous-time Markov chain is recurrent, then its tail~$\sigma$-algebra is trivial. This is the content of Orey's theorem \cite{BlackwellFreedman64}; note that this theorem is originally stated for discrete-time chains, but the proof works equally for continuous-time chains. Then, if~$G$ is recurrent, the tail~$\sigma$-algebra of a simple random walk on $G$ is trivial.
    \item Krishnapur and Peres presented in \cite{PeresKrishnapur2004} an example of a recurrent graph with the finite collision property: the comb lattice Comb($\mathbb Z$), obtained by starting with~$\mathbb Z^2$ and deleting all horizontal edges except for those with vertical coordinate zero. With minor adaptations, their proof also shows that Comb($\Z$) has the continuous finite collision property. Hence, the set of extremal measures of the voter model on~Comb($\Z$) is~$\{ \mu_\alpha: \alpha \in [0,1] \}$. 
    \item The assumption that the tail $\sigma$-algebra $\mathcal{T}$ of a random walk is trivial is necessary in~\textbf{(B)}. To see that the statement would be false without this assumption, we show that for the voter model on the 3-regular tree, we have
\begin{equation}\label{Contr0}
    \mathcal{I}_e \supsetneq \{ \mu_\alpha: \alpha \in [0,1] \}.
    \end{equation}
We do this by exhibiting a measure~$\mu \in \mathcal I$ such that
    \begin{equation}\label{Contr1}
        \mu \notin \overline{\text{conv}}(\{ \mu_\alpha: \alpha \in [0,1] \}),
    \end{equation}
    where~$\overline{\text{conv}}$ denotes the closed convex hull of a set of measures. Noting that Theorem~1.8.c of \cite{Liggett2005} states that~$\overline{\text{conv}}(\mathcal{I}_e) = \mathcal{I}$,
    we see that~\eqref{Contr1} gives~\eqref{Contr0}.   
    
    In order to construct~$\mu$, fix an arbitrary vertex of the tree and let~$L$ be one of the tree branches emanating from it. 
    Let $\bar \eta \in \{0,1\}^V$ be given by $\bar \eta(x)=1$ if and only if $x \in L$. For every $A \in \mathcal{P}_\mathrm{fin}(V)$, by duality we have
    \begin{equation*}
P_{\bar \eta}(\eta_t \equiv 1 \text{ on } A) = P_A(A_t \subset L) = P_A( Y^1_t \in L, \dots, Y^{|A|}_t \in L),    
    \end{equation*}
    where~$(Y^1_t,\ldots,Y^{|A|}_t)$ is a system of coalescing random walks. As~$t \to \infty$, the indicator of the event~$\{Y^1_t,\ldots, Y^{|A|}_t \in L\}$ converges almost surely to the indicator of the event~$\{\exists s_0: Y^1_s,\ldots, Y^{|A|}_s \in L \;\forall s \ge s_0 \}$. Hence, 
    \[\lim_{t \to \infty} P_{\bar \eta}(\eta_t \equiv 1 \text{ on } A) = P_A(\exists s_0: Y^1_s,\ldots, Y^{|A|}_s \in L \;\forall s \ge s_0).\]
    This shows that~$\eta_t$ converges weakly, as~$t \to \infty$, to a distribution~$\mu$ satisfying
    \[\mu(\eta: \eta \equiv 1 \text{ on } A) = P_A(\exists s_0: Y^1_s,\ldots, Y^{|A|}_s \in L \;\forall s \ge s_0).\]
Since $\mu$ is obtained as a weak limit of the law of the process as time is taken to infinity, it is stationary. To check~\eqref{Contr1}, it is easy to check that any measure~$\nu \in \overline{\text{conv}}(\{ \mu_\alpha: \alpha \in [0,1] \}$ is invariant under tree isomorphisms, and on the other hand, for any vertex~$x$ of the tree,~$\mu(\{\eta: \eta(x) = 1\})$ is equal to the probability that a random walk started from~$x$ eventually never leaves~$L$; this probability clearly depends on~$x$.
 
\end{enumerate}
\begin{proof}[Proof of Theorem \ref{MetaTheorem}:]
We compare a system of coalescing random walks with a system of independent random walks. Given $k \in \N$, let~$(X^i_t)_{t \ge 0}$, $i = 1,\dots,k$, be independent continuous-time random walks on $G$ which jump away from any site with rate 1 and choose jump destinations uniformly at random among neighbours of the starting point of the jump. We write 
\begin{equation}\label{definition_cal_X}
    \X_t := \{X^1_t, \dots, X^k_t\}, \; t \ge0.    
\end{equation}
Given that $X^i_0 = Y^i_0$ for each $i$, there is a natural coupling of~$(Y^1_t,\dots,Y^k_t)$ and~$(X^1_t,\dots,X^k_t)$ so that $A_t \subseteq \X_t$ for all~$t \geq 0$. This can be done by using the same transition times for the two processes, but at each coalescence, retaining only one of the coalescing particles in~$A_t$. Using this coupling, it follows that for any $f:\mathcal{P}_\mathrm{fin}(V) \to \R$  with~$|f| \le 1$, 
\begin{equation}\label{CouplingEquation}
    \big|E_A [ f(A_t)] - E_A [f(\X_t) ] \big| \leq P_A( \exists t \geq 0: |A_t| < |A_0| ) =: g(A).
\end{equation} 

1. Assume $G$ has the continuous infinite collision property, then 
\begin{equation}\label{prob1_collision}
    P_{x,y}(\exists t\geq0:X^1_t=X^2_t)=1 \text{ for every } x,y \in V. 
\end{equation}
Fix $\mu \in \mathcal{I}$. For any $x,y \in V$, we have
\begin{equation}\label{Dim1MetaTheorem}
    \begin{aligned}
    \mu( \eta: \eta(x) \neq \eta(y) ) & =  \int P_\eta( \eta_t(x) \neq \eta_t(y) ) \; \mu(\mathrm{d}\eta) \\[.2cm]
    & \hspace{-1.5mm} \stackrel{(\ref{DualityMetaTheorem})}{=} \int P_{\{x,y\}}( \eta(Y^1_t) \neq \eta_t(Y^2_t) ) \; \mu(\mathrm{d}\eta) \\[.2cm]
    & \leq P_{\{x,y\}}( Y^1_t \neq Y^2_t \hspace{1mm} ) \stackrel{(*)}{=} P_{x,y}(X^1_s \neq X^2_s \; \forall s \in [0,t]  ),
    \end{aligned}
\end{equation}
where $(*)$ follows from the coupling of the coalescing and independent random walks. By taking the limit as $t \to \infty$ in (\ref{Dim1MetaTheorem}) and using (\ref{prob1_collision}) we see that $\mu$ is supported on the consensus configurations, so $\mu = \alpha \delta_{\overline{1}} + (1-\alpha)\delta_{\overline{0}}$ for some $\alpha \in [0,1]$. This proves that~$\mathcal{I}_e =~\{\delta_{\overline{0}},\delta_{\overline{1}}\}$. 

2. Assume $G$ has the continuous finite collision property. It is easy to see that
\begin{equation}\label{eq_FunctionGMetaTheorem}
    \lim_{t \to \infty}E_A [g(\X_t)] = \lim_{t \to \infty} E_A[g(A_t)] = 0 \quad \text{ for any } A \in \mathcal{P}_\mathrm{fin}(V).
\end{equation}

Given a measure~$\kappa$ on~$\{0,1\}^V$, we will write
\[\widehat{\kappa}(A):= \kappa(\{\eta: \eta \equiv 1 \text{ on } A\}),\quad A \in \mathcal P_{\mathrm{fin}}(V).\]

Fix~$\mu \in \mathcal I$ for the rest of the proof. We divide the proof into steps.\\[-.3cm]

\noindent \textbf{Step 1: Definition of the measure~$\nu$.} For any~$A \in \mathcal P_{\mathrm{fin}}(V)$, by the Markov property we have
\[
E_A [\widehat{\mu}(\mathcal X_{s+t})] = E_A [ E_{\mathcal X_s} [\widehat{\mu}(\mathcal X_t)]].
\]
Moreover, by stationarity of~$\mu$ and duality,
\[
E_A[\widehat{\mu}(\mathcal X_s)] = E_A[E_{\mathcal X_s}[\widehat{\mu}(A_t)]].
\]
Together with~\eqref{CouplingEquation} and~\eqref{eq_FunctionGMetaTheorem}, this gives
\begin{equation}\label{eq_first_twostep}
|E_A [\widehat{\mu}(\mathcal X_{s+t})] - E_A[\widehat{\mu}(\mathcal X_s)]| \le E_A[g(\mathcal X_s)] \xrightarrow{s \to \infty} 0.
\end{equation}
This readily implies that
\begin{equation}\label{eq_limit_exists} E_A[\widehat{\mu}(\mathcal X_t)] \text{ converges as~$t \to \infty$}.\end{equation}
Using the inclusion-exclusion formula, it is easy to see that for every~$t$ there exists a unique probability~$\nu_t$ on~$\{0,1\}^V$ such that~$\widehat{\nu_t}(A)=E_A[\widehat{\mu}(\mathcal X_t)]$. Then,~\eqref{eq_limit_exists} implies that~$\nu_t$ converges weakly, as~$t \to \infty$, to a measure~$\nu$, which by compactness of~$\{0,1\}^V$ is also a probability, and satisfies
\begin{equation}
    \label{eq_nu_satisfies}
    \widehat{\nu}(A) = \lim_{t \to \infty} E_A[\widehat{\mu}(\mathcal X_t)],\quad A \in \mathcal P_{\mathrm{fin}}(V).
\end{equation}

We will need a bound involving~$\mu$ and~$\nu$. Similarly to what we did in~\eqref{eq_first_twostep}, for any~$A$ and any~$s,t$ we have
\begin{equation*}
    |\widehat{\mu}(A) - E_A[E_{A_s}[\widehat{\mu}(\mathcal X_t)]]| = |E_A[E_{A_s}[\widehat{\mu}(A_t)]]- E_A[E_{A_s}[\widehat{\mu}(\mathcal X_t)]]| \le E_A[g(A_s)].
    \end{equation*}
By taking~$t \to \infty$ in the above and using the bounded convergence theorem with the convergence~$E_{A_s}[\widehat{\mu}(\mathcal X_t)] \xrightarrow{t \to \infty} \widehat{\nu}(A_s)$, we obtain
\begin{equation}\label{eq_before_finetti}
    |\widehat{\mu}(A) - E_A[\widehat{\nu}(A_s)]| \le E_A[g(A_s)].
\end{equation}

\noindent \textbf{Step 2:~$\nu$ is a mixture of Bernoulli product measures.} It follows from~\eqref{eq_nu_satisfies} and the Markov property that, for any~$A \in \mathcal P_{\mathrm{fin}}(V)$, under~$P_A$, the process~$\widehat{\nu}(\mathcal X_t)$,~$t \ge 0$, is a bounded martingale. By the martingale convergence theorem, there exists a random variable $Z_A$ such that~$P_A$-almost surely,~$\widehat{\nu}(\mathcal X_t)$ converges to~$Z_A$.

We now want to use the assumption of triviality of the tail~$\sigma$-algebra of the random walks to obtain that~$Z_A$ is almost surely constant, but we need to be a bit careful, because we are dealing with multiple random walks ($|A|$ many of them), not just one. Let
$$ \mathcal{T}^i = \bigcap_{t \geq 0} \sigma(X^i_s: s \geq t), \quad i =1,\dots,k.$$
Under the assumption that each $\mathcal{T}^i$ is trivial, we obtain that $\sigma( \mathcal{T}^1, \dots, \mathcal{T}^k )$ is also trivial. Letting~$\mathcal G:= \bigcap_{t \geq 0} \sigma( \mathcal{T}^1_t, \dots, \mathcal{T}^k_t )$, Lemma 2 of~\cite{LindvallRogers1986} implies that~$\mathcal G$ is the same as~$\sigma( \mathcal{T}^1, \dots, \mathcal{T}^k )$, and then this~$\sigma$-algebra is trivial. Note that the result of~\cite{LindvallRogers1986} is given for two processes, but the results for $k$ processes readily follows. Since~$Z$ is~$\mathcal{G}$-measurable, it is constant~$P_A$-almost surely, and, since martingales have constant expectation,
\begin{equation}\label{eq_q_limit_as_k}
    E_A[Z_A] = Z_A = \widehat{\nu}(A) \quad P_A\text{-almost surely}.
\end{equation}

Now let $A,B \in \mathcal P_\mathrm{fin}(V)$ with $|A|=|B|$. Let $\tau_B = \inf\{ t \ge 0: \X_t = B \}$. Using~\eqref{eq_q_limit_as_k} and the Strong Markov property, we readily obtain
 \begin{align*}
    \widehat{\nu}(A) \cdot P_A(\tau_B< \infty)  = \widehat{\nu}(B) \cdot P_A( \tau_B < \infty).
\end{align*}
By irreducibility of the random walks, we have that $P_A(\tau_B< \infty) >0$, so we conclude that~$\widehat{\nu}(A) = \widehat{\nu}(B)$. 

This proves that~$\widehat{\nu}(A)$ only depends on~$A$ through~$|A|$, which (again using inclusion-exclusion) readily implies that~$\nu$ is exchangeable. It then follows from de Finetti's theorem that there exists a (uniquely determined) probability~$\gamma$ on the Borel sets of~$[0,1]$ such that
\begin{equation}
    \label{eq_after_finetti}
    \nu = \int_0^1 \pi_\alpha \; \gamma(\mathrm{d}\alpha).
\end{equation}

\noindent \textbf{Step 3:~$\mu$ is a mixture of the measures~$\mu_\alpha$.} 
By putting together~\eqref{eq_before_finetti} and~\eqref{eq_after_finetti}, and also using Fubini's theorem we have, for any~$A$ and~$s$:
\begin{equation}\label{eq_consequence_finetti}    \left|\widehat{\mu}(A) - \int_0^1 E_A[\widehat{\pi_\alpha}(A_s)] \; \gamma(\mathrm d \alpha)\right| \le E_A[g(A_s)].\end{equation}

Recall from~\eqref{eq_dual_Tt} that the measures~$\mu_\alpha$ satisfy
\[\widehat{\mu_\alpha}(A) = \lim_{s \to \infty} E_A[\widehat{\pi_\alpha}(A_s)]. \]
This gives
\[
\int_0^1 \widehat{\mu_\alpha}(A)\;\gamma(\mathrm d \alpha) = \lim_{s \to \infty} \int_0^1 E_A[\widehat{\pi_\alpha}(A_s)]\;\gamma(\mathrm d \alpha).
\]
Hence, we can take~$s \to \infty$ on both sides of the inequality~\eqref{eq_consequence_finetti}, also using~\eqref{eq_FunctionGMetaTheorem}, to obtain that
\begin{equation}\label{eq_decompose_alpha}
\mu = \int_0^1 \mu_\alpha\;\gamma(\mathrm d \alpha).
\end{equation}

\noindent \textbf{Step 4: Conclusion.} It follows readily from~\eqref{eq_decompose_alpha} that~$\mathcal{I} \subseteq \overline{\mathrm{conv}}( \{ \mu_\alpha: \alpha \in [0,1] \} )$. The reverse inclusion holds because~$\mathcal I$ is closed and convex. Recall that the Krein-Milman theorem asserts that if $\mathcal{B} \subseteq \mathcal{I}$ and $\mathcal{I} = \overline{\mathrm{conv}} (\mathcal{B})$, then $\mathcal{I}_e \subseteq \mathrm{cl}(\mathcal{B})$, where~$\mathrm{cl}$ denotes the closure of a set of measures. Therefore, 
$$ \mathcal{I}_e \subseteq \mathrm{cl}( \{\mu_\alpha: \alpha \in [0,1]\} ) = \{\mu_\alpha: \alpha \in [0,1]\},$$
where the equality holds because the family $\{\mu_\alpha: \alpha \in [0,1]\}$ is closed. 

Finally, it remains to prove that $\mu_\alpha$ is extremal for every~$\alpha$ (that is, it cannot be written as~$\mu_\alpha = \int_0^1 \mu_\beta\; \gamma(\mathrm d \beta)$  for some~$\gamma \neq \delta_{\alpha}$). 

To prove this, fix two distinct measures~$\gamma,\gamma'$ on~$[0,1]$ and let~$\mu = \int \mu_\alpha\gamma(\mathrm{d}\alpha)$ and~$\mu' = \int \mu_\alpha \gamma'(\mathrm{d}\alpha)$. Take the measures~$\nu$ and~$\nu'$ corresponding to~$\mu$ and~$\mu'$, respectively, as in Step 1 of the proof. By inspecting the proof, it is easy to check that~$\nu=\int \pi_\alpha \gamma(\mathrm{d}\alpha)$ and~$\nu'=\int \pi_\alpha \gamma'(\mathrm{d}\alpha)$, and from this, using  for instance the uniqueness in de Finetti's theorem, we conclude that~$\nu \neq \nu'$, so~$\mu \neq \mu'$.
\end{proof}

\section{The voter model with stirring}\label{IVMSection}
This section will use the graph notation introduced at the beginning of Section~\ref{background}. Throughout this section, we fix an integer $d\ge 2$ and let $G = (V,E)$ be a $d$-regular graph.

\subsection{Duality and stationary measures}
Given~$\mathsf{v} \in [0,\infty)$, the voter model with stirring on~$G=(V,E)$ with parameter~$\mathsf{v}$, denoted by~$\xi_t$,~$t \geq 0$, is the Feller process with state space $\Omega = \{0,1\}^V$ and pre-generator
$$\mathcal{L}_\mathrm{stir} f(\xi) = \mathsf{v} \cdot \sum_{ \{x,y\}\in E } \frac{f(\xi^{x,y}) - f(\xi)}{d} + \sum_{ \substack{x,y \in V: \\ x \sim y} }   \frac{f(\xi^{y \to x}) - f(\xi)}{d},$$
where $f:\Omega \to \R$ is any function that only depends on finitely many coordinates,~$\xi \in \Omega$, 
$$\xi^{x,y}(z) = \left\{ \begin{array}{ll}
    \xi(z) & \text{if } z \neq x,y,  \\
    \xi(x) & \text{if } z = y, \\
    \xi(y) & \text{if } z = x, \end{array} \right.$$
    and~$\xi^{y \to x}$ is as in~\eqref{eq_notation_conf}.
The first part of the pre-generator interchanges the voter opinions at rate~$\mathsf{v}/d$ and the second part is the usual voter model. Given~$\xi \in \Omega$, we denote by~$P_\xi$ a probability measure under which~$(\xi_t)_{t \ge 0}$ is defined and satisfies~$P_\xi(\xi_0 = \xi)=1$. 

Now, we construct a dual Markov chain for the voter model with stirring. To do so, we introduce the following families of Poisson point processes on $[0,\infty)$:
$$\{ \mathcal T^{(x,y)}: x \sim y \} \text{ each at rate } 1/d \quad \text{and} \quad \{ \mathcal H^{\{x,y\}}: x \sim y \} \text{ each at rate } \mathsf{v}/d.$$
We denote by $\Prob$ a probability measure under which all these processes are defined and are independent. For each $x \in V$, we then define the \textbf{coalescing-stirring random walks}~$(Y^x_t)_{t \geq 0}$ as follows: set $Y^x_0 = x$ and assume that $Y^x_{t^-}$ has been defined,
\begin{itemize}
    \item if $t \in  \mathcal T^{ (Y^x_{t^-},z) }$ for some $z \sim Y^x_{t^-}$ or $t \in \mathcal H^{ \{ Y^x_{t^-},z \} } $ for some $z \sim Y^x_{t^-}$, then $Y^x_t = z$;
    \item if $t \notin \big (\cup_{z \sim Y^x_{t^-}} \mathcal T^{ (Y^x_{t^-},z) } \big) \cup \big( \cup_{z \sim Y^x_{t^-}} \mathcal H^{ (Y_{t^-},z) } \big)$, then $Y^x_t = Y^x_{t^-}$.
\end{itemize}
From the definition above, it is clear that each $(Y^x_t)_{t \geq 0}$ is a continuous-time random walk on~$V$ which jumps away from any site with rate $1+\mathsf{v}$ and chooses jump destinations uniformly at random among neighbors of the starting point of the jump. When~$Y^x_t$ and~$Y^y_t$ are nearest neighbors they can swap their positions or meet (in the latter case they coalesce). 

We write $S_t(A) := \{Y^x_t: x \in A\}$ for every $t \geq 0$. It is easy to see that~$S_t(A)$ is a Markov chain on~$\mathcal{P}_\mathrm{fin}(\Z^d)$ that follows the rates:
\begin{equation}\label{rates_stirring_set_walker}
    \begin{array}{lll}
        \vspace{1mm} q(A, (A \backslash \{x\})\cup \{y\} ) & = \tfrac{1}{d}(\mathsf{v}+1), & x \in A, y \notin A, x \sim y \\
        q(A,A \setminus \{x\}) & = \displaystyle \sum_{ \substack{y \in A: y \sim x} } 1/d, & x \in A.
    \end{array}
\end{equation}
 Given~$A \in \mathcal{P}_\mathrm{fin}(V)$, we denote by $P_A$ a probability measure under which $(S_t)_{t \ge 0} = (S_t(A))_{t \ge 0}$ is defined and satisfies $P_A(S_0 = A) = 1$.

For a configuration~$\xi \in \Omega$ and~$A \in \mathcal{P}_\mathrm{fin}(V)$, we write~$\phi(\xi,A) = \prod_{x \in A}\xi(x)$. We also write~$\phi_\xi(A) = \phi_A(\xi) = \phi(\xi,A)$. Letting~$\mathcal{L}^d_\mathrm{stir}$ be the generator corresponding to the chain as in~\eqref{rates_stirring_set_walker}, it is easy to see that
$$ \mathcal{L}_\mathrm{stir}\phi_A(\xi) = \mathcal{L}^d_\mathrm{stir}\phi_\xi(A).$$
Then we have the dual relationship
\begin{equation}\label{duality_equa_stir}
    P_\xi( \xi_t \equiv 1 \text{ on } A ) = P_A( \xi \equiv 1 \text{ on } S_t ).
\end{equation}
Proving analogously as in Section \ref{MetaTheoremSection}, we have a family of distinct stationary measures~$\mu^\mathrm{stir}_\alpha$,~$\alpha \in [0,1]$, satisfying 
$$\mu^\mathrm{stir}_\alpha( \xi: \xi \equiv 1 \text{ on } A ) = E_A \big[ \alpha^{|S_\infty|} \big] \hspace{1mm} \text{ for all } A \in \mathcal{P}_\mathrm{fin}(V). $$

\subsection{Characterization of the stationary measures}\label{sec_charac_stat_stir}
To obtain a characterization of the stationary measures of the voter model with stirring, we compare a system of coalescing-stirring random walks with independent random walks.\\[-.3cm]

\noindent \textbf{Coupling of Markov chains.} Let~$S$ be a countable set and $r_1$ and $r_2$ be bounded jump rates for two continuous-time Markov chains on~$S$. Under a probability measure $\widehat{\Prob}$, we construct a Markov chain $(W_t,Z_t,I_t)_{t \ge 0}$ on the state space
$$ \{ (a,a,0): a \in S \} \cup \{ (a,b,1): a,b \in S \},$$
starting from $(c,c,0)$. The jump rates of this coupled chain are given as follows. From a point of the
form $(a,b,1)$ the chain jumps to $(a',b,1)$ with rate $r_1(a,a')$, and to $(a,b',1)$ with rate $r_2(b,b')$. From a point of the form $(a, a, 0)$, the rate depends on the destination, as follows:
\begin{itemize}
    \item if $r_1(a, a') = r_2(a, a')$, then the chain jumps to $(a', a', 0)$ with rate $r_1(a, a')$;
    \item if $r_1(a, a') \neq r_2(a, a')$, then the chain jumps to $(a', a, 1)$ with rate $r_1(a, a')$ and to $(a, a', 1)$ with rate $r_2(a, a')$.
\end{itemize}
It is easy to see that the marginal rates for $W_t$ and $Z_t$ are correct, they are continuous-time random walks with jumping rates $r_1$ and $r_2$, respectively. Let 
$$ \tau_\mathrm{coup} := \inf\{ t \ge 0: I_t = 1 \}, $$
the time when the coupling breaks. The process
$$ M_t = \mathds{1}\{\tau_\mathrm{coup} \leq t \} - \int^t_0 \sum_{b \hspace{0.5mm} : \hspace{0.5mm} r_1(W_s,b) \neq r_2(W_s,b)} \big( r_1(W_s,b) + r_2(W_s,b) \big) \cdot \mathds{1}\{ s < \tau_\mathrm{coup} \} \; \mathrm{d}s, \hspace{1mm} t\ge 0 $$
can be seen to be a martingale. Using the martingale property and taking the limit as~$t \to \infty$, we have that
\begin{equation}\label{coup_upper_bound}
    \widehat{\Prob}( \tau_\mathrm{coup} < \infty ) = \widehat{\E} \Big[ \int^{\tau_\mathrm{coup}}_0 \sum_{b} \big(r_1(W_t,b) + r_2(W_t,b) \big) \cdot \mathds{1}\{ r_1(W_t,b) \neq r_2(W_t,b) \} \; \mathrm{d}t \Big],
\end{equation}
Note that we could use $Z_t$ instead of $W_t$ in the previous equation.\\[-.3cm]

\noindent \textbf{Application.} We want to apply the above coupling with $r_1$ being the rates for the coalescing-stirring random walks $(S_t)_{t \ge 0}$ and $r_2$ the rates independent random walks~$(\X_t)_{t \ge 0}$ as in~\eqref{definition_cal_X} with jump rate~$1+\mathsf v$ . For any~$A \in \mathcal P_\mathrm{fin}$ and~$|f| \leq 1$,
\begin{equation}\label{Coupling_ineq_SandX}
    \tfrac12\big|E_A[f( S_t )] - E_A[f(\X_t)] \big| \le P_A( \exists t \ge 0: S_t \neq \X_t ) \leq P_A( \tau_\mathrm{coup} < \infty) =: g_\mathrm{stir}(A).
\end{equation}
It will be important to compare $g_\mathrm{stir}$ with~$g$ given in \eqref{CouplingEquation}. We note that for $A \in \mathcal{P}_\mathrm{fin}(V)$,
\begin{equation}\label{rate_ineq_Psi}
    \sum_{B \in \mathcal{P}_\mathrm{fin}(V) } \big( r_1(A,B) + r_2(A,B) \big) \cdot \mathds{1}\{ r_1(A,B) \neq r_2(A,B) \} \leq 3(1+\mathsf v)\cdot \Phi(A),
\end{equation}
where $\Phi(A) = \sum_{e = \{x,y\} \in E(\Z^d) } \mathds{1} \big\{ x,y \in A \big\}.$  By~\eqref{coup_upper_bound} and~\eqref{rate_ineq_Psi}, we have
\begin{equation}\label{g_stir_upper_bound}
    g_\mathrm{stir}(A) \leq 3(1+\mathsf{v}) \cdot E_A \Big[ \int^{\tau_\mathrm{coup}}_0 \Phi(S_t) \Big] \stackrel{(\ast)}= 3(1+\mathsf{v}) \cdot E_A \Big[ \int^{\tau_\mathrm{coup}}_0 \Phi(\X_t) \Big] \text{ for any } A \in \mathcal{P}_\mathrm{fin}(V),
\end{equation}
where $(\ast)$ holds by the observation given below~\eqref{coup_upper_bound}. Now we prove Theorem \ref{ThmCharacterizationIVM}.

\subsubsection{Proof of Theorem \ref{ThmCharacterizationIVM}}

1. Assume $G$ has the continuous infinite collision property. Let $S_t = \{Y^1_t,Y^2_t\}$, $t\ge 0$, be the Markov chain~\eqref{rates_stirring_set_walker}. We will prove that~$(Y^1_t)$ and~$(Y^2_t)$ coalesce a.s. for any initial location~$Y^1_0, Y^2_0$. Then, following similar ideas from Theorem~\ref{MetaTheorem} we conclude that the only extremal stationary measures of the voter model with stirring are the consensus ones, which correspond to $\mu^\mathrm{stir}_0$ and $\mu^\mathrm{stir}_1$.

Consider the coupling with independent random walks. By the continuous infinite collision property, $Y^1_t$ and $Y^2_t$ will be neighbors a.s. Then they will not be neighbours anymore if one of them jumps apart or they coalesce, the former we call rejection. After a rejection, the walkers will eventually become neighbours. To avoid coalescing we need to reject infinitely many times. Since the rejection event has a positive probability (uniformly over any pair of vertices of the graph) not coalescing is not possible. Therefore
$$P_{x,y}( \exists t \geq 0: Y^1_t = Y^2_t ) = 1 \quad \text{for any } x,y \in V.$$ 

\noindent 2. Assume $G$ has the continuous finite collision property. We aim to follow the steps given in the proof of Theorem~\ref{MetaTheorem}. The only remaining task is to prove~\eqref{eq_FunctionGMetaTheorem} for $g_\mathrm{stir}$, after which the four steps will follow immediately.

First, we prove that 
\begin{equation}\label{exp_g_stirr_X_t}
    \lim_{t \to \infty}E_A [ g_\mathrm{stir}(\X_t) ] = 0 \quad \text{for any } A \in \mathcal{P}_\mathrm{fin}(V).
\end{equation}
Let $\tau_\mathrm{coll}$ be the first time when two random walks in $\X_t$ collide, that is, 
$$ \tau_\mathrm{coll} = \inf\{t \ge 0 : |\X_t|<|\X_0| \}. $$
Using a martingale argument as we did when we construct the coupling of Markov chains, it can be seen that 
\begin{equation}\label{tau_coll_equality}
    P_A( \tau_\mathrm{coll} < \infty ) = 2(1+\mathsf{v}) \cdot E_A \Big[ \int^{\tau_\mathrm{coll}}_0 \Phi(\X_s) \; \mathrm{d}s \Big].
\end{equation}
We observe that $\tau_\mathrm{coup} \leq \tau_\mathrm{coll}$, then
\begin{equation}\label{exp_g_stirr_to_0}
    \begin{aligned}
        E_A [g_\mathrm{stir}(\X_t)] & \stackrel{\eqref{g_stir_upper_bound}}\leq 3(1+\mathsf{v}) \cdot E_A \Big[ E_{\X_t} \Big[ \int^{\tau_\mathrm{coup}}_0 \Phi(\X_s)\; \mathrm{d}s  \Big]  \Big] \\
        & \hspace{1.2mm} \leq 3(1+\mathsf{v}) \cdot E_A \Big[ E_{\X_t} \Big[ \int^{\tau_\mathrm{coll}}_0 \Phi(\X_s)\; \mathrm{d}s  \Big]  \Big] \\
        & \stackrel{\eqref{tau_coll_equality}}= \tfrac32 \cdot E_A [ P_{\X_t}( \tau_\mathrm{coll} < \infty ) ] = \tfrac32 \cdot P_A( \exists s \ge t: |\X_s| < |\X_t|  ).
    \end{aligned}
\end{equation}
The finite collision property of $G$ and~\eqref{exp_g_stirr_to_0} implies~\eqref{exp_g_stirr_X_t}. 


Second, we aim to obtain
\begin{equation}\label{lim_g_stir_St}
    \lim_{t \to \infty}E_A[g_\mathrm{stir}(S_t)] = 0.
\end{equation}
Note that in the standard case, this follows from the monotonicity of $g$ and the coupling, which ensures that the set of coalescing random walks is always included in the set of independent random walks at any time. In this case, we still have that $g_\mathrm{stir}$ is monotone but our coupling only ensures the inclusion until $\tau_\mathrm{coup}$. 

Let $\tau_\mathrm{coal}$ be the first time when there is a coalescence in $S_t$, that is,
$$ \tau_\mathrm{coal} = \inf\{ t \ge 0: |S_t| < |S_0| \}. $$
Using a martingale argument as we did when we construct the coupling of Markov chains,
it can be seen that
\begin{equation}\label{tau_coal_equality}
    P_A( \tau_\mathrm{coal} < \infty ) = 2 \cdot E_A \Big[ \int^{\tau_\mathrm{coal}}_0 \Phi(S_s) \; \mathrm{d}s \Big].
\end{equation}
Next, we can easily prove that
\begin{equation}\label{coal_counting_limit}
     \lim_{t \to \infty} E_A[ \text{number of coalescences of } (S_s)_{s \ge t} \mid (S_r)_{0 \leq r \leq t} ] = 0 \quad \text{ for any } A \in \mathcal{P}_\mathrm{fin}(V).
\end{equation}
As a consequence
\begin{equation}\label{exp_St_limit}
    \begin{aligned}
    E_{S_t}\Big[ \int^{\tau_\mathrm{coal}}_0 \Phi(S_s) \; \mathrm{d}s \Big] & \stackrel{\eqref{tau_coal_equality}}= \tfrac12 \cdot P_{S_t}(\tau_\mathrm{coal}<\infty) \\[.2cm]
    & \vspace{5mm} \hspace{2.2mm }= \tfrac12 \cdot P_A( \exists s \ge t: |S_s| < |S_t| \mid (S_r)_{0 \leq r \leq t} ) \\[.2cm]
    & \hspace{2.2mm}\leq \tfrac12 \cdot E_A[ \text{number of coalescences of } (S_s)_{s \ge t} \mid (S_r)_{0 \leq r \leq t} ]. 
    \end{aligned}
\end{equation}
Hence, by~\eqref{coal_counting_limit} and~\eqref{exp_St_limit}, we have
\begin{equation}\label{exp_int_coal_St}
    \lim_{t \to \infty} E_{S_t}\Big[ \int^{\tau_\mathrm{coal}}_0 \Phi(S_s) \; \mathrm{d}s \Big] = 0 \hspace{2mm} P_A\text{-a.s.} \hspace{2mm} \text{ for any } A \in \mathcal{P}_\mathrm{fin}(V).
\end{equation}
We observe that $\tau_\mathrm{coup} \leq \tau_\mathrm{coal}$, then
\begin{equation}\label{exp_g_stir_st_bound}
    \begin{aligned}
        E_A[g_\mathrm{stir}(S_t)] & \stackrel{\eqref{g_stir_upper_bound}}\leq 3(1+\mathsf{v}) \cdot E_{S_t} \Big[ \int^{\tau_\mathrm{coup}}_0 \Phi(S_s) \; \mathrm{d}s \Big]\\
    & \hspace{1.4mm} \leq 3(1+\mathsf{v}) \cdot E_{S_t} \Big[ \int^{\tau_\mathrm{coal}}_0 \Phi(S_s) \; \mathrm{d}s \Big].
    \end{aligned}
\end{equation}
Hence~\eqref{lim_g_stir_St} follows from~\eqref{exp_int_coal_St} and~\eqref{exp_g_stir_st_bound}. Therefore, from~\eqref{exp_g_stirr_X_t},~(\ref{lim_g_stir_St}) and the procedure presented in the proof of Theorem \ref{MetaTheorem}, we have that~$\mathcal{I}_e = \{\mu^\mathrm{stir}_\alpha: \alpha \in [0,1]\}$.

\subsection{Ergodicity of the stationary measures in the Euclidean lattice}\label{sec_ergo_stat_stir}
In this section, we prove Proposition~\ref{prop_intro_ergo}. Given a vertex~$x \in \Z^d$, we denote by~$|x|_1$ its~$\ell^1$ norm. We write
$$B_1(L) = \{ x \in \Z^d: |x|_1 \leq L \}, \hspace{3mm} B_1(x,L) = x + B_1(L). $$
For $A,B \in \mathcal{P}_\mathrm{fin}(\Z^d)$ we write 
$$\mathrm{dist}(A,B) = \min \{ |x-y|_1:  x \in A,y \in B \}.$$ 
For~$x \in \Z^d$, we define the shift transformation~$\tau_x$ on~$\{0,1\}^{\Z^d}$ by $\tau_x (\eta)(y) = \eta(y-x).$ for~$y \in \Z^d$. This shift map induces the shift translation of subsets of~$\{0,1\}^{\Z^d}$. In particular, for~$A \in \mathcal{P}_\mathrm{fin}(\Z^d)$,
\begin{equation*}
    \tau_x(\{ \eta: \eta \equiv 1 \text{ on } A \}) = \{ \eta: \eta \equiv 1 \text{ on } A+x \},
\end{equation*}
where $A + x = \{y + x: y \in A\}$. 

It is easy to see that the $d$-Euclidean lattice has the infinite collision property for $d=1,2$, then by Theorem~\ref{ThmCharacterizationIVM} the only extremal stationary measures for the voter model with stirring are the consensus ones which are easily seen to be spatial ergodic. For $d \ge 3$, the $d$-Euclidean lattice has the finite collision property and the tail $\sigma$-algebra of a simple random walk is trivial, then by Theorem~\ref{ThmCharacterizationIVM} the set of extremal stationary measures for the voter model with stirring is precisely the family $\mu^\mathrm{stir}_\alpha$, $\alpha \in [0,1]$. To prove the ergodicity of these measures we use the next classical result concerning independent random walks. 
\begin{lemma}\label{last_lemma}
Fix $L \ge 0$ and $x,y \in \Z^d$ such that $|x-y|_1 \geq L$. Let $(X^1_t)_{t \geq 0}$ and $(X^2_t)_{t \geq 0}$ be independent simple random walks. Then
$$ P_{x,y} (\exists t \geq 0: |X^1_t - X^2_t|_1 \leq L ) \leq c|B_1(L)| \cdot \big(|x-y|_1-L\big)^{2-d}, $$
where $c$ is a dimension-dependent constant.
\end{lemma}


For every $\alpha \in [0,1]$, the measure $\mu^\mathrm{stir}_\alpha$ is easily seen to be spatial invariant. Now we prove that $\mu^\mathrm{stir}_\alpha$ is mixing, which implies ergodicity. Let $A,B \in \mathcal{P}_\mathrm{fin}(\Z^d)$, we define the Markov chains $(S^1_t)_{t \geq 0}$, $(S^2_t)_{t \geq 0}$ and $(S^3_t)_{t \geq 0}$ on $\textbf{P}$ with jumping rates \eqref{rates_stirring_set_walker} such that the following holds: 
\begin{itemize}
    \item $S^1_0 =A$, $S^2_0=B$, $S^3_0 = A \cup B$;
    \item $(S^1_t)$ and $(S^2_t)$ are independent;
    \item $(S^3_t)$ is defined as follows: let $\sigma := \inf\{t \geq 0: \mathrm{dist}(S^1_t,S^2_t)=1\}$,
\begin{enumerate}
    \item for $t \leq \sigma$, $S^3_t = S^1_t \cup S^2_t$,
    \item for $t > \sigma$, $S^3_t$ moves independently of $S^1_t$ and $S^2_t$. 
\end{enumerate}
\end{itemize}
Then
\begin{align*}
    \big| \mu^\mathrm{stir}_\alpha (\xi: \xi \equiv 1 \text{ on } A & \cup B ) - \mu^\mathrm{stir}_\alpha( \xi: \xi \equiv 1  \text{ on } A   ) \cdot \mu^\mathrm{stir}_\alpha( \xi: \xi \equiv 1 \text{ on } B  ) \big| \\
    & = \textbf{E} \big[ \alpha^{|S^3_\infty|} \big]  - \textbf{E} \big[ \alpha^{|S^1_\infty|} \big] \cdot \textbf{E} \big[ \alpha^{|S^2_\infty|} \big] = \textbf{E} \big[ \alpha^{|S^3_\infty|} - \alpha^{|S^1_\infty| + |S^2_\infty|} \big] \\
    & \leq \textbf{E}( \sigma < \infty ) \stackrel{(\ast)}\leq \sum_{u \in A, v \in B} P_{u,v}( \exists t \geq 0: |X^1_t - X^2_t|_1=1 ),
\end{align*}
where $(\ast)$ holds by the coupling with independent random walks. Next, we take~$B = B + x$. Hence, the mixing property
$$\lim_{|x|_1 \to \infty }\mu^\mathrm{stir}_\alpha( \xi: \xi \equiv 1 \text{ on } A \cup (B+x) ) - \mu^\mathrm{stir}_\alpha( \xi: \xi \equiv 1 \text{ on } A ) \cdot \mu^\mathrm{stir}_\alpha( \xi: \xi \equiv 1 \text{ on } B ) =  0$$
follows from an application of Lemma \ref{last_lemma} with $L = 1$.

\section{Voter model on dynamical percolation}\label{VMonDPSection}
The spaces of configuration are denoted by~$\Omega_\mathrm{site} = \{0,1\}^{\Z^d}$ and~$\Omega_{\mathrm{edge}} = \{0,1\}^{E(\Z^d)}$. \\[-.3cm]

\noindent \textbf{Voter model with range $R$.} Given $d,R\in \N$, the voter model $(\eta_t)_{t \geq 0}$ with range $R$ on the $d$-dimensional Euclidean lattice is a Feller process on $\Omega_\mathrm{site}$ with Markov pre-generator
\begin{equation*}
    \mathcal{L}_{\text{vm}} f(\eta) = \sum_{ \substack{x,y \in \Z^d: \\ 1 \leq |x-y|_1 \leq R} } \frac{f(\eta^{y \to x}) - f(\eta)}{ |B_1(R)|-1 },
\end{equation*}
where $f:\Omega_\mathrm{site} \to \R$ is any function that only depends on finitely many coordinates and~$\eta \in \Omega_\mathrm{site}$. 

As a consequence of Theorem \ref{MetaTheorem}, for~$d=1,2$, the only extremal stationary distributions are the consensus ones $\delta_{\bar{0}}$ and $\delta_{\bar{1}}$, and for $d \geq 3$, there is a family $\{\mu_\alpha$, $\alpha \in [0,1]\}$ of stationary measures which coincide with the set of extremal stationary distributions. \\[-.3cm]

\noindent \textbf{Dynamical percolation.} The basic bond percolation model on $(\Z^d,E(\Z^d))$ is obtained by fixing $p \in [0,1]$ and sampling the full configuration of edges from 
$$\pi^\mathrm{edge}_p := \bigotimes_{e \in E(\Z^d)}\text{Ber}(p).$$

Dynamical percolation is a dynamical variant of the basic percolation model. In this model, to each edge in $E(\Z^d)$, there corresponds a two-state (open, closed) continuous-time Markov chain (these chains being independent for different edges). The entire configuration of open and closed edges at time $t$, denoted by $\zeta_t$, is an element of~$\Omega_\mathrm{edge}$. Given~$\mathsf{v} \in (0,\infty)$, the evolution of the dynamical percolation~$(\zeta_t)$ at speed~$\mathsf{v}$ is a Markov process with pre-generator 
$$\mathcal{L}_\mathrm{dp}f(\zeta) = \sum_{e \in E(\Z^d)}c(e,\zeta,\mathsf{v}) \cdot [ f(\zeta^e) - f(\zeta) ],$$
where $f:\Omega_\mathrm{edge} \to \R$ is any function that only depends on finitely many coordinates,~$\zeta~\in~\Omega_\mathrm{edge}$,
$$ \zeta^e(e') = \left\{ \begin{array}{ll}
    \zeta(e') & \text{if } e' \neq e \\
    1-\zeta(e) & \text{if } e' = e
\end{array} \right. \quad \text{and} \quad c(e,\zeta,\mathsf{v}) = \left\{ \begin{array}{ll}
    \mathsf{v} \cdot p, & \text{if } \zeta(e)=0 \\
    \mathsf{v} \cdot (1-p), & \text{if } \zeta(e)=1
\end{array} \right. $$
This process is stationary with respect to $\pi^\mathrm{edge}_p$. 

In defining the voter model on dynamical percolation, we need to introduce the following definition. The indicator function of an event or set is denoted by $\mathds{1}_{\{ \cdot \}}$ or $\mathds{1}\{ \cdot \}$.

Given  $x,y \in \Z^d$ and $\zeta \in \Omega_\mathrm{edge}$, we write $x \stackrel{\zeta}{\leftrightarrow} y$ if there is a sequence of vertices~$\gamma(0), \dots, \gamma(n)$ so that $\gamma(0) = x$, $\gamma(n)=y$ and for every $i$, $\{\gamma(i) ,\gamma(i+1) \} \in E(\Z^d)$ and~$\zeta( \{\gamma(i) ,\gamma(i+1) \} ) = 1$. When the sequence of vertices is contained in $A \subset \Z^d$, we write~$x \stackrel{\zeta}{\leftrightarrow} y$ inside~$A$.


Given $d,R \in \N$, $\mathsf{v} \in (0,\infty)$ and $p \in [0, 1]$, the voter model on dynamical percolation on~$(\Z^d,E(\Z^d))$ with range~$R$, environment speed~$\mathsf{v}$ and edge density~$p$, denoted by~$M_t = (\eta_t,\zeta_t)$,~$t \geq 0$, is the Feller process with state space $\Omega_\mathrm{site} \times \Omega_\mathrm{edge}$ and pre-generator: 
$$\mathcal{L}_\mathrm{vmdyn}f(\eta,\zeta) = \sum_{\substack{x,y \in \Z^d: \\ 1 \leq |x-y|_1 \leq R}} \frac{f(\eta^{y \to x},\zeta) - f(\eta,\zeta)}{ |B_1(R)|-1 }\cdot \mathds{1}\{ x \stackrel{\zeta}{\leftrightarrow} y \text{ inside } B_1(x,R) \} + \mathcal{L}_\mathrm{dp}f_\eta(\zeta), $$
where $f: \Omega_\mathrm{site} \times \Omega_\mathrm{edge} \to \R$ is any function that only depends on finitely many coordinates,~$(\eta,\zeta) \in \Omega_\mathrm{site} \times \Omega_\mathrm{edge}$, and $f_\eta(\zeta) = f(\eta,\zeta)$. It can be checked that this spin system is well-defined for any choice of parameters.

As a system of coalescing random walks is a dual for the voter model, it is natural to expect that a system of \textbf{coalescing random walks on dynamical percolation} will serve as a dual for the voter model on dynamical percolation. We define the process and prove this duality in Section \ref{duality_stat_measures_vmodyn_section}. In Section \ref{section_characterization_stat_measures} we use this duality relationship to study the stationary measures for~$(M_t)$, the key aspect of this analysis is the understanding of the collision of \textbf{independent random walks on dynamical percolation}. 

\subsection{Random walks on dynamical percolation}\label{RandomWalksonDPSection}
In this section, we study random walks on dynamical percolation. Analogously to what was done in Section \ref{MetaTheoremSection}, we aim to couple a system of coalescing random walks with a system of independent random walks. This coupling will be done in Section \ref{section_characterization_stat_measures}.

Given $k \in \N$, for $\mathbf{x} = (x_1,\dots,x_k) \in (\Z^d)^k$ and $y \in \Z^d$, we write
\begin{equation}\label{x_bf_i_y}
    \mathbf{x}^{i,y} = (x^{i,y}_1, \dots, x^{i,y}_k), \; \text{ where } x^{i,y}_j = \left\{ \begin{array}{cc}
    x_j & \text{ if } i \neq j  \\
    y & \text{ if } i = j. 
\end{array} \right.
\end{equation}
A system of~$k$ \textbf{independent random walks on a single dynamical percolation environment} is a Markov process~$(X^1_t,\ldots, X^k_t, \zeta_t)_{t \ge 0}$ on~$(\mathbb Z^d)^k \times \Omega_\mathrm{edge}$ with generator
\begin{equation}\label{Gen_irwdyn}    
\begin{aligned}
    &\mathcal L_{\mathrm{irwdyn}}f(\mathbf{x},\zeta) = \mathcal L_{\mathrm{dp}}f_{\mathbf{x}}(\zeta) \\
    &+ \frac{1}{|B_1(R)|-1}\sum_{i=1}^k \sum_{y \in B_1(x_i,R)} \mathds{1}\{x_i \xleftrightarrow{\zeta } y \text{ inside }B_1(x_i,R)\} \cdot \big(f_\zeta(\mathbf{x}^{i,y}) - f_\zeta(\mathbf{x}) \big),
\end{aligned}
\end{equation}
where $f: (\Z^d)^k \times \Omega_\mathrm{edge} \to \R$ is any function such that $f_\mathbf{x}(\zeta) = f(\mathbf{x},\zeta)$ only depends on finitely many coordinates and $f(\mathbf{x},\zeta) = f_\mathbf{x}(\zeta) = f_\zeta(\mathbf{x})$. 

We have the following results for independent random walks on a single dynamical percolation environment.
\begin{proposition}\label{CollisionRWd12}
Let $p \in (0,1], \mathsf{v} >0$ and $R \in \N$. In dimensions~$1$ and~$2$, two independent random walks (started from arbitrary positions) with range~$R$ on a single dynamical percolation environment (started from an arbitrary configuration and with parameters $p$ and $\mathsf{v}$) meet almost surely.
\end{proposition}
The following proposition shows that in dimensions 3 and higher, two independent random walks with range~$R$ on a single dynamical percolation environment have a positive probability of never meeting.
\begin{proposition}\label{Bigtheorem}
Let $d \geq 3$, $p \in [0,1]$, $\mathsf v > 0$ and $R \in \N$. For any~$\ell \ge 0$ and~$\epsilon > 0$, there exists~$L > 0$ such that the following holds. Let~$\zeta \in \Omega_\mathrm{edge}$,~$x,y \in \mathbb Z^d$ with~$|x-y|_1 > L$, and let~$(X^1_t,X^2_t,\zeta_t)_{t \ge 0}$ be a system of two independent random walks on a single dynamical percolation environment with parameters~$p,\mathsf v,R$, started from~$X^1_0 = x$,~$X^2_0 =y$, and~$\zeta_0 = \zeta$. Then,
\begin{equation*}
    \mathbb P(\exists t \ge 0:\; |X^1_t - X^2_t|_1 \le \ell) < \epsilon.
\end{equation*}
\end{proposition}

\begin{corollary}\label{CoroBigThm} 
Let $d \geq 3$, $p \in (0,1]$, $\mathsf v > 0$ and $R \in \N$. Let~$(X^1_t,X^2_t,\zeta_t)_{t \ge 0}$ be a system of two independent random walks on a single dynamical percolation environment with parameters~$p,\mathsf{v},R$. Then, for any $\ell \geq 0$,
\begin{equation*}
\lim_{t \to \infty} \sup_{\zeta \in \Omega_\mathrm{edge}}\sup_{x,y \in \Z^d} \Prob \big( \exists s \geq t:  |X_s^1-X_s^2|_1 \leq \ell \big) = 0. 
\end{equation*}

\end{corollary}
To prove these results, it will be convenient to introduce a graphical construction for independent random walks on dynamical percolation. In order to do so, we first introduce notation. Given $R \in \N$, we define an \textbf{instruction manual} for a random walk on~$\mathbb Z^d$ with range~$R$ as a pair~$(\mathcal T, \mathsf m)$, where~$\mathcal T$ is a Poisson point process on~$[0,\infty)$ with intensity~1, and~$\mathsf m: \mathcal T \to B_1(R)\backslash \{0\}$ is such that, conditionally on~$\mathcal T$, the values~$\{\mathsf m(t): \; t \in \mathcal T\}$ are independent and uniformly distributed on~$B_1(R) \backslash \{0\}$. The idea here is of course that~$\mathcal T$ encodes the times when the walker attempts to jump, and the marks~$\mathsf m(t)$ encode the spacial displacements associated to the jumps. We sometimes abuse notation and denote the pair~$(\mathcal T,\mathsf m)$ by~$\mathcal T$.

Given $p \in [0,1]$ and $\mathsf{v} > 0$, to construct a dynamical percolation~$(\zeta_t)_{t \geq 0}$ with density~$p$ and speed~$\mathsf{v}$ on $\Z^d$, we independently assign to every~$e \in E(\Z^d)$ a pair~$(\mathcal{U}^e, \phi^e)$, where~$\mathcal{U}^e$ is a Poisson process on $[0,\infty)$ with intensity $\mathsf{v}>0$ and $\phi^e$ is a family of i.i.d. random variables~$\{\phi^e_t : t \in [0,\infty)\}$ with distribution~Ber($p$). We gather all of these processes in $\mathcal{U}$. Given an initial configuration~$\zeta_0 \in \Omega_\mathrm{edge}$, for every $e \in E(\Z^d)$ and $t \geq 0$, we set
$$\zeta_t(e) = \left\{ \begin{array}{lll}
        \zeta_0(e) & \text{ if }  \mathcal{U}^e \cap [0,t] = \varnothing, & \\
        \phi_r^e  & \text{ if }  \mathcal{U}^e \cap [0,t] \neq \varnothing, & r = \max\{\mathcal{U}^e \cap [0,t]\}.
\end{array} \right.$$ 
Given a graphical construction $\mathcal{U}$ for dynamical percolation $(\zeta_t)_{t \ge 0}$, we say that this process is \textbf{started from stationarity} if $\zeta_0$ is a random edge configuration with distribution~$\pi^\mathrm{edge}_p$ (of course, whenever the initial edge configuration is random, we assume that it is independent of the graphical construction $\mathcal{U}$).

Fix~$R \in \N$,~$p \in [0,1]$ and~$\mathsf v > 0$. Let~$\mathcal U$ be a graphical construction for dynamical percolation $(\zeta_t)$ with density~$p$ and speed~$\mathsf v$ on~$\mathbb Z^d$, and let~$\mathcal T$ be an instruction manual for a random walk with range~$R$ on~$\mathbb Z^d$ ($\mathcal T$ and~$\mathcal U$ are always taken independent). Also fix~$\zeta_0 \in \Omega_{\mathrm{edge}}$ and~$x \in \mathbb Z^d$.
We now define the \textbf{random walk flow}
\begin{equation*}
    \mathbb X_t(\zeta,x) = \mathbb X_t^{\mathcal U, \mathcal T}(\zeta, x),\quad t \ge 0.
\end{equation*}
The intuitive idea is that the walker attempts to jump as dictated by its instruction manual, but it is only allowed to perform a jump, say from position~$w$ to position~$w + \mathsf m(t)$, if the percolation configuration at time~$t$ contains an open path between these two points, and this path stays within distance~$R$ from~$w$. Let us now give the formal definition. Enumerate~$\mathcal T = (t_1, t_2, \ldots)$ in increasing order. Set~$\mathbb X_t(\zeta,x) = x$ for~$t \in [0,t_1)$. For~$n \in \mathbb N$, assume that~$w=\mathbb X_{t_n^-}(\zeta,x)$ has been defined, and set, for~$t \in [t_n,t_{n+1})$:
\[
\mathbb X_{t}(\zeta, x) = \begin{cases}
    w+\mathsf m(t_n) & \text{if } w \xleftrightarrow{\zeta_{t_n}} w+\mathsf{m}(t_n) \text{ inside }B_1(w,R);\\[.1cm]
    w&\text{otherwise.}
\end{cases}
\]

Given a  graphical construction~$\mathcal U$ for dynamical percolation and~$\zeta \in \Omega_{\mathrm{edge}}$, points~$x_1,\ldots, x_k \in \Z^d$ and independent instruction manuals~$\mathcal T_1,\ldots, \mathcal T_k$, the processes
\begin{equation*}
(\mathbb X_t^{\mathcal U,\mathcal T_1}(\zeta,x_1))_{t \ge 0},\quad \dots, \quad (\mathbb X_t^{\mathcal U,\mathcal T_k}(\zeta,x_k))_{t \ge 0}
\end{equation*}
are independent random walks on a single dynamical percolation environment, that is, their law coincides with the process with the generator given earlier. Note that the initial positions~$x_1,\ldots, x_k$ are not necessarily distinct.

Proposition~\ref{CollisionRWd12} is proved in Section~\ref{CollisionRWd12SectionProof}, Proposition~\ref{Bigtheorem} and Corollary~\ref{CoroBigThm} in Section~\ref{SectionRegenrationTimes}. The proofs rely on a regeneration time technique that is presented in Section~\ref{SectionRegenrationTimes}.

\subsubsection{Dimension 1 and 2}\label{CollisionRWd12SectionProof}
In this section, we prove Proposition \ref{CollisionRWd12}. Fix $p \in (0,1]$, $\mathsf{v} > 0$ and $R \in \N$. Let~$\zeta \sim \pi^\mathrm{edge}_p$,~$x,y \in \Z^d$, and let~$(X^1_t,X^2_t,\zeta_t)_{t \ge 0}$ be a system of two independent random walks on a single dynamical percolation environment with parameters $p,\mathsf{v},R$ started from~$X^1_0 = x$,~$X^2_0 = y$ and~$\zeta_0 = \zeta$. Following the ideas of Corollary 1.3 of \cite{Hutchcroft2022} with minor modifications, we obtain that for any~$x,y \in \Z^d$,
\begin{equation}\label{CollisionRWd12Eq1}
    \Prob \big( \exists t \geq 0: X_t^1 = X^2_t \big) = 1.
\end{equation}
The remainder of this section is dedicated to introducing tools for extending statements from the environment starting at stationarity, such as \eqref{CollisionRWd12Eq1}, to any initial edge configuration. To do so, we introduce some notation. We write
$$ E_1(x,L) = \{ e = \{x,y\} \in E(\Z^d): x,y \in B_1(x,L) \}.$$
For $A \subseteq \Z^d$ and $E \subseteq E(\Z^d)$ we write 
$$\mathrm{dist}(A,E) = \min \{ |x-y|_1:  x \in A, \; y \in \cup_{e \in E}\{e\} \}.$$

\begin{lemma}\label{ArgumentCouplingRW}
For any~$d,R \in \N$,~$p \in [0,1]$,~$k \in \N$,~$\mathsf{v}>0$ and~$\epsilon > 0$ there exists~$r > 0$ such that the following holds. Let~$\mathcal U$ be a graphical construction for dynamical percolation with parameters~$p$ and~$\mathsf v$, and let~$\mathcal T_1, \ldots, \mathcal T_k$ be instruction manuals for random walks with range~$R$. Then, for any~$x_1,\ldots, x_k$ and any~$\zeta, \zeta' \in \Omega_\mathrm{edge}$ which agree in~$\cup^k_{i=1} E_1(x_i,r)$, we have
\begin{equation}\label{ArgumentCouplingRW_eq}
    \Prob\big( \mathbb{X}^{\mathcal U, \mathcal T_i}_t(\zeta,x_i) = \mathbb{X}^{\mathcal U,\mathcal T_i}_t(\zeta', x_i) \text{ for all } t \geq 0   \text{ and } i \in  \{1,\dots,k\}\big) > 1 - \epsilon.
\end{equation}
\end{lemma}

\begin{proof}
For every $n \geq 1$, we define the events
$$G_n = \{ \exists e \in E(\Z^d): \mathrm{dist}(\{x_1,\dots,x_k\},\{e\}) \leq n \text{ and } \mathcal{U}^e \cap [0, \sqrt{n} ] = \varnothing \}.$$
and
$$H_n = \{ R \cdot |\{\text{attempted steps by the random walks by time } \sqrt{n}\}| \geq n \}.$$
It is easy to see that 
$$\sum_n \Prob(G_n)< \infty \hspace{3mm} \text{ and } \hspace{3mm} \sum_n \Prob(H_n)< \infty, $$
so there exists $N$ such that
$$\Prob\big[ \big(\cap_{n \geq N} G^c_n \big) \cap \big( \cap_{n \geq N} H^c_n \big) \big] \geq 1 - \epsilon.$$
Set $r = N$. The event inside the probability says that the edges outside $\cup^k_{i=1} E_1(x_i,r)$ refresh before the random walks reach them. This implies that the trajectory of the random walks remains the same as long as their initial environments are identical within~$\cup^k_{i=1} E_1(x_i,r)$.
\end{proof} 
Given the configurations $\zeta,\xi \in \Omega_\mathrm{edge}$ and $E \in \mathcal{P}_\mathrm{fin}(E(\Z^d))$, we define the configuration~$\xi^{\zeta \to E}$ by
\begin{equation}\label{XiEzeta}
    \xi^{\zeta \to E}(e) = \left\{ \begin{array}{cl}
        \zeta(e) & \text{ if } e \in E   \\
        \xi(e) & \text{ if } e \in E(\Z^d) \setminus E.  
    \end{array} \right.
\end{equation}

\begin{proof}[Proof of Proposition \ref{CollisionRWd12}] We assume that the process is built from a graphical construction. 
Fix $x,y \in \Z^d$ and $\zeta \in \Omega_\mathrm{edge}$. Due to \eqref{CollisionRWd12Eq1}, for $\zeta_0 \sim \pi$, 
\begin{equation}\label{D12Eq1}
    \Prob \big( \exists t \geq 0: \mathbb{X}_t^{\mathcal{U},\mathcal{T}_1}(\zeta_0,x) = \mathbb{X}_t^{\mathcal{U},\mathcal{T}_2}(\zeta_0,y) \big) = 1.
\end{equation}
For any $E \in \mathcal{P}_\mathrm{fin}( E(\Z^d) )$, we consider $\zeta^{\zeta \to E}_0$ as in \eqref{XiEzeta}, then
\begin{equation}\label{D12xi}
    \begin{aligned}
    \Prob( \exists t \geq 0: \mathbb{X}^{\mathcal{U},\mathcal{T}_1}_t( & \zeta_0^{\zeta \to E},x) = \mathbb{X}^{\mathcal{U},\mathcal{T}_2}_t(\zeta_0^{\zeta \to E},y) \big) \\
    & = \Prob( \exists t \geq 0: \mathbb{X}_t^{\mathcal{U},\mathcal{T}_1}(\zeta_0,x) = \mathbb{X}^{\mathcal{U},\mathcal{T}_2}_t(\zeta_0,x) \mid \zeta_0 \equiv \zeta \text{ on } E \big) \stackrel{\eqref{D12Eq1}}= 1.         
    \end{aligned}
\end{equation}
Given~$\epsilon > 0$, we choose~$r > 0$ as in Lemma \ref{ArgumentCouplingRW}, so that, letting~$E_r=E_1(x,r) \cup E_1(y,r)$, we have
\begin{align*}
    \Prob(  \mathbb{X}^{\mathcal{U},\mathcal{T}_1}_t &  (\zeta,x) \neq \mathbb{X}^{\mathcal{U},\mathcal{T}_2}_t(\zeta,y) \;\forall t \geq 0) \\[.2cm]
    & \stackrel{\eqref{ArgumentCouplingRW_eq}}\leq \Prob \big(  \mathbb{X}^{\mathcal{U},\mathcal{T}_1}_t( \zeta_0^{\zeta \to E_r},x) \neq \mathbb{X}^{\mathcal{U},\mathcal{T}_2}_t(\zeta_0^{\zeta \to E_r},y) \;\forall t \geq 0\big) + \epsilon \stackrel{(\ref{D12xi})} = \epsilon.
\end{align*}
\end{proof}

\subsubsection{Dimension 3 and higher}\label{CouplingRWSection}
In this section and the next, we prove Proposition \ref{Bigtheorem} and Corollary \ref{CoroBigThm}. To accomplish this, we couple independent random walks defined on a single environment with random walks defined on their own individual and independent environments. 

Fix a graphical construction $\mathcal{U}$ for dynamical percolation,  $\zeta \in \Omega_\mathrm{edge}$,~$x_1,\dots,x_k \in \Z^d$, and independent instructions manuals $\mathcal{T}_1,\dots,\mathcal{T}_k$. Recall the definition of the flows~$(\mathbb X_t^{\mathcal U,\mathcal T_i}(\zeta,x_i))_{t \ge 0}$ from the beginning of Section~\ref{RandomWalksonDPSection}; these give independent random walks started at~$x_1,\ldots, x_k$, all sharing the same dynamical percolation environment.

Abbreviating~$\mathbf{x}=(x_1,\ldots,x_k)$ and~$\mathbf T =(\mathcal T_1,\ldots,\mathcal T_k)$, we now define the processes
\[(\mathbb A_t^{\mathcal U,\mathbf T}(\zeta,\mathbf x))_{t \ge 0},\quad (\mathbb B_t^{\mathcal U,\mathbf T}(\zeta,\mathbf x))_{t \ge 0},\]
both taking values in~$\mathcal P_\mathrm{fin}(\Z^d)$. We first explain these processes intuitively. Suppose that, at time 0, we have no knowledge of the environment (which is encoded by~$\zeta_0$). As time progresses, our knowledge of the environment will change:~$\mathbb A_t^{\mathcal U,\mathbf T}(\zeta,\mathbf x)$ will denote the set of edges which we know to be open at time~$t$, and $\mathbb B_t^{\mathcal U,\mathbf T}(\zeta,\mathbf x)$ the set of edges which we know to be closed at time~$t$ (in particular, these two sets will always be disjoint). Whenever an edge updates, we remove it from our knowledge (either from~$\mathbb A_t^{\mathcal U,\mathbf T}(\zeta,\mathbf x)$ or from~$\mathbb B_t^{\mathcal U,\mathbf T}(\zeta,\mathbf x)$). Moreover, whenever a random walk tries to jump (say, from position~$x$ to position~$y$ at time~$t$), we reveal the status of all edges of~$E_1(x,R)$ at time~$t$ to decide if the jump is allowed; we collect the edges that have been revealed as open and closed in~$\mathbb A_t^{\mathcal U,\mathbf T}(\zeta,\mathbf x)$ and~$\mathbb B_t^{\mathcal U,\mathbf T}(\zeta,\mathbf x)$, respectively.

We now turn to the formal definition. To abbreviate, we write here $$(\mathbb{A}_t, \mathbb{B}_t) = (\mathbb{A}_t^{\mathcal U, \mathbf{T}}(\zeta,\mathbf{x}), \mathbb{B}_t^{\mathcal U, \mathbf{T}}(\zeta,\mathbf{x})), \hspace{2mm} \text{ for all } t \ge 0.$$
Set $\mathbb{A}_0 = \mathbb{B}_0 = \varnothing$, and assume that~$w_i = \mathbb{X}_{t^-}(\zeta_0,x_i)$,~$i=1,\dots,k$, $\mathbb{A}_{t-}$ and~$\mathbb B_{t-}$ have been defined, 
\begin{itemize}
    \item if $t \in \cup^k_{i=1}\mathcal{T}_i$, say $t \in \mathcal{T}_j$, then 
     $$\mathbb{A}_t = \mathbb{A}_{t^-} \cup \{\text{open edges of } E_1(w_j,R)\} \text{ and } \mathbb{B}_t =\mathbb{B}_{t^-} \cup \{\text{closed edges of } E_1(w_j,R)\};$$
    \item if $t \in \mathcal{U}^e$, then $\mathbb{A}_t = \mathbb{A}_{t^-} \setminus \{e\}$, $\mathbb{B}_t = \mathbb{B}_{t^-} \setminus \{e\}$;
    \item if $t \notin \big( \cup^k_{i=1} \mathcal{T}_i \big) \cup \big( \cup_{e \in E(\Z^d)}\mathcal{U}^e \big)$, then $\mathbb{A}_t = \mathbb{A}_{t^-}$ and $\mathbb{B}_t = \mathbb{B}_{t^-}$.
\end{itemize}
In case~$k=1$ (when we have only one walker started from~$x$ with instruction manual~$\mathcal T$), we write~$\mathbb A_t^{\mathcal U,\mathcal T}(\zeta,x)$ and~$\mathbb B_t^{\mathcal U,\mathcal T}(\zeta,x)$. Note that, for~$\mathbf T = (\mathcal T_1,\ldots, \mathcal T_k)$ and~$\mathbf x = (x_1,\ldots, x_k)$, we have
\begin{equation*}
    \mathbb A_t^{\mathcal U, \mathbf T}(\zeta, \mathbf x) = \bigcup_{i=1}^{k}\mathbb A_t^{\mathcal U, \mathcal T_i}(\zeta, x_i),\qquad \mathbb B_t^{\mathcal U, \mathbf T}(\zeta, \mathbf x) = \bigcup_{i=1}^{k}\mathbb B_t^{\mathcal U, \mathcal T_i}(\zeta, x_i). 
\end{equation*}
\begin{remark}\label{remark_mc_xab}
    In case $\zeta_0 \sim \pi^\mathrm{edge}_p$, it is easy to see that the process
\begin{equation}\label{Markov_chain_same_environment}
    \big(\mathbb{X}^{\mathcal{U},\mathcal{T}_1}_t(\zeta_0,x_1) ,\;\dots,\;\mathbb{X}^{\mathcal{U},\mathcal{T}_k}_t(\zeta_0,x_k),\; \mathbb{A}^{\mathcal{U},\mathbf{T}}_t(\zeta_0,\mathbf{x}), \;\mathbb{B}^{\mathcal{U},\mathbf{T}}_t(\zeta_0,\mathbf{x}) \big)
\end{equation}
is a Markov chain started from~$(x_1,\dots,x_k,\varnothing,\varnothing)$. 
\end{remark}
We construct the mentioned coupling for two independent random walks. From the same reasoning of Remark~\ref{remark_mc_xab}, we see that if $\zeta_0 \sim \pi^\mathrm{edge}_p$, the process
$$ \big( \mathbb{X}^{\mathcal{U},\mathcal{T}_1}_t(\zeta_0,x_1) ,\; \mathbb{A}^{\mathcal{U},\mathcal{T}_1}_t(\zeta_0,x_1), \;\mathbb{B}^{\mathcal{U},\mathcal{T}_1}_t(\zeta_0,x_1), \; \mathbb{X}^{\mathcal{U},\mathcal{T}_2}_t(\zeta_0,x_2) ,\; \mathbb{A}^{\mathcal{U},\mathcal{T}_2}_t(\zeta_0,x_2), \;\mathbb{B}^{\mathcal{U},\mathcal{T}_2}_t(\zeta_0,x_2) )_{t \ge 0} $$
is a Markov chain. We denote by $r_\mathrm{single}$ its jump rate function. Let $(X^i_t,A^i_t, B^i_t)_{t \ge 0}$, $i=1,2$, be independent copies of \eqref{Markov_chain_same_environment} with $k=1$. We denote by $r_\mathrm{separate}$ the jump rate function of the Markov chain
$$(X_t, A^1_t, B^1_t, Y_t, A^2_t, B^2_t)_{t \ge 0}.$$
To state the main results of this section we introduced the following functions. For any~$\ell \in (0, \infty)$ and~$x,y \in \Z^d$, we write
\begin{equation}\label{g_function_x,y}
    \begin{aligned}
        & g_\ell(x,y) := \E \Big[ \int^\infty_0 \mathds{1}\{ \min \big( |X_t - Y_t|_1, \hspace{0.5mm} \mathrm{dist}(X_t,A^2_t \cup B^2_t), \hspace{0.5mm} \mathrm{dist}(Y_t,A^1_t \cup B^1_t) \big) \leq \ell \}  \; \mathrm{d}t \Big] \\[.2cm]
        & f_\ell(x,y) := \Prob \big( \exists t \geq 0:  \min \big( |X_t - Y_t|_1, \hspace{0.5mm} \mathrm{dist}(X_t,A^2_t \cup B^2_t), \hspace{0.5mm} \mathrm{dist}(Y_t,A^1_t \cup B^1_t) \big) \leq \ell \big).
    \end{aligned}    
\end{equation}
\begin{lemma}\label{couple_single_separate}
    For any $d,R \in \N$, $p \in [0,1]$, $\mathsf{v}>0$ and~$x,y \in \Z^d$, there exists a coupling~$(\mathbf{X}_t, \mathbf{Y}_t)_{t \ge 0}$ under $\widehat{\Prob}$ such that~$(\mathbf{X}_t)$ and~$(\mathbf{Y}_t)$ are continuous-time Markov chains with jump rate function~$r_\mathrm{sinlge}$ and~$r_\mathrm{separate}$, respectively, started both from~$(x,\varnothing,\varnothing,y,\varnothing,\varnothing)$ and
    $$ \widehat{\Prob}( \exists t \ge 0: \mathbf{X}_t \neq \mathbf{Y}_t ) \leq 2g_{2R}(x,y).$$
\end{lemma}
As a simple corollary, we have 
\begin{lemma}\label{lemma_main}
For any $d,R \in \N$, $p \in [0,1]$, $\mathsf{v}>0$ and~$x,y \in \Z^d$ with~$|x-y|_1 > 2R$ the following holds. Write
\begin{equation}\label{abbr_not_X1_X2_same_environment}
    \begin{aligned}
    & X^1_t = \mathbb{X}^{\mathcal U, \mathcal T_1}(\zeta_0,x), \hspace{2mm} F^1_t = \mathbb{A}^{\mathcal U, \mathcal T_1}(\zeta_0,x) \cup \mathbb{B}^{\mathcal U, \mathcal T_1}(\zeta_0,x); \\[0.1cm]
    & X^2_t = \mathbb{X}^{\mathcal U, \mathcal T_2}(\zeta_0,y), \hspace{2mm} F^2_t = \mathbb{A}^{\mathcal U, \mathcal T_2}_t(\zeta_0,y) \cup \mathbb{B}^{\mathcal U, \mathcal T_2}_t(\zeta_0,y),\hspace{2mm} t \ge 0,
\end{aligned}    
\end{equation}
Then,
\begin{equation}
    \Prob(  \exists t\ge 0: \min \big( |X^1_t - X^2_t|_1, \hspace{0.5mm} \mathrm{dist}(X_t,F^2_t), \hspace{0.5mm} \mathrm{dist}(Y_t,F^1_t) \big) \leq \ell ) \leq f_\ell(x,y) + 2g_{2R}(x,y).
\end{equation}
\end{lemma}
We emphasize that, although random walks in \emph{separate environments} are used in the definitions of the functions~$g$ and~$f_\ell$, we then ``forget'' this and regard these functions as ``black boxes'' in the above lemma, which is a statement about random walks \emph{in a single environment}. In Section~\ref{SectionRegenrationTimes}, we well see that $f_\ell$ and $g_\ell$ converge to~0 as~$|x-y|_1 \to \infty$ for any $\ell$ via \textbf{regeneration times} (Lemma \ref{lemma_limit_fl_g}). To prove Lemma \ref{lemma_main} we use the following lemma.

To prove Lemma~\ref{couple_single_separate} we use the following results concerning couplings of Markov chains.
\begin{lemma}\label{lemma_pre_main}
    Let $S$ be a countable set and $B \subseteq S$. Let $r_1$ and $r_2$ be bounded rates for continuous-time Markov chains on~$S$. Assume that $r_1(s,s') = r_2(s,s')$ for all $s \in S \setminus B$ and $s' \in S$. Then, there exists a coupling $(W_t,Z_t)_{t \geq 0}$ under $\textbf{P}$ such that $(W_t)$ and $(Z_t)$ are continuous-time Markov chains started both from $z \in S$ with rates $r_1$ and $r_2$, respectively, and
    $$ \textbf{P}( \exists t \geq 0 : W_t \neq Z_t ) \leq \textbf{E}\Big[ \int^\infty_0 r_2(Z_s,B) \Big] \;\mathrm{d}s, $$
    where $r_2(Z_s,B) = \sum_{m \in B}r_2(Z_s,m)$.
\end{lemma}
\begin{proof}
In a probability space $\mathbf{P}$, we define the Markov chains $W = (W_t)_{t \ge 0}$ and $Z = (Z_t)_{t \ge 0}$ that are coupled so that 
\begin{itemize}
    \item $W$ and $Z$ have jumping rates $r_1$ and $r_2$, respectively, and $W_0 = Z_0 = z$;    
    \item let $\tau_B = \inf\{ s \ge 0: Z_s \in B \}$,
    \begin{enumerate}
        \item for $t \leq \tau_B$, $W$ and $Z$ move together;
        \item for $t > \tau_B$, $W$ and $Z$ move independently.
    \end{enumerate}
\end{itemize}
Since these Markov chains agree until the first time $Z$ hits the set $B$, we see that the process
$$M_t = \mathds{1}\{ \tau_B \leq t \} - \int^{t }_0 r_2(Z_s,B) \cdot \mathds{1} \{ s < \tau_B\} \;\mathrm{d}s, \hspace{3mm} t \geq 0, $$
is a martingale. Then, for any $t \ge 0$,
$$ 0 = M_0 = \textbf{E}[M_t] = \textbf{P}(\tau_B \leq t) - \textbf{E}\Big[ \int^t_0 r_2(Z_s,B) \cdot \mathds{1} \{ s < \tau_B\} \mathrm{d}s \Big].$$
Hence,
$$ \textbf{P}( \exists t\ge 0: W_t \neq Z_t ) \leq \textbf{P}( \tau_B < \infty ) \leq \textbf{E} \Big[ \int^\infty_0 r_2(Z_s,B) \mathrm{d}s \Big].$$
\end{proof}
\begin{proof}[Proof of Lemma \ref{couple_single_separate}]
Fix $x,y \in \Z^d$. We set 
\begin{align*}
    &B = \{ (u,O_1,C_1,v,O_2,C_2) \in S:\; \min(|u-v|_1,\;\mathrm{dist}(u,O_2 \cup C_2),\; \mathrm{dist}(v,O_1 \cup C_1)) \leq R \},\\[.2cm]
        &B' = \{ (u,O_1,C_1,v,O_2,C_2) \in S:\; \min(|u-v|_1,\;\mathrm{dist}(u,O_2 \cup C_2),\; \mathrm{dist}(v,O_1 \cup C_1)) \leq 2R \}.
\end{align*}
It is easy to see that $r_\mathrm{single}(s,s') = r_\mathrm{separate}(s,s')$ for any $s \notin B$ and any $s'$. By Lemma~\ref{lemma_pre_main}, there exists a coupling $(\textbf{X}_t,\textbf{Y}_t)_{t \geq 0}$ under~$\widehat{\Prob}$ of chains~$(\mathbf{X}_t)$ and~$(\mathbf{Y}_t)$ with rates~$r_\mathrm{single}$ and~$r_\mathrm{separate}$, respectively, such that~$\mathbf{X}_0 = \mathbf{Y}_0 = (x,\varnothing,\varnothing,y, \varnothing, \varnothing)$ and
\begin{equation}\label{coupling_proof_1}
    \widehat{\Prob}( \exists t \geq 0: \mathbf{X}_t \neq \mathbf{Y}_t ) \leq \widehat{\E} \Big[ \int^\infty_0 r_2(\mathbf{Y}_s,B)  \mathrm{d}s \Big].
\end{equation}
We write $\mathbf{Y}_t = (U_t,O^1_t,C^1_t,V_t,O^2_t,C^2_t)$ for all $t \geq 0$. By inspecting the jumping rates, we can check that
\begin{equation}\label{sum_rate_coupling}
    r_2(\mathbf{Y}_t,B) \le 2\cdot \mathds{1} \big\{\mathbf{Y}_t \in B'\} = 2 \cdot \mathds{1} \{ \min \big( | U_t -  V_t|_1, \hspace{0.5mm}  \mathrm{dist}(U_t,O^2_t \cup C^2_t), \hspace{0.5mm} \mathrm{dist}(V_t, O^1_t \cup C^1_t) \big) \leq 2R \big\}.
\end{equation} 
\end{proof}

\subsubsection{Regeneration times}\label{SectionRegenrationTimes}
Regeneration times are extensively used in the context of random walks in random environments. This technique was initially introduced in the context of dynamical percolation in \cite{PeresStaufferSteif2015}. We will define regeneration times for two independent random walks, each running on its own dynamical percolation environment (these two environments also being independent).

We recall some notation introduced earlier. Let $\mathcal{U}_1$ and $\mathcal{U}_2$ be graphical constructions for dynamical percolation $(\zeta^1_t)_{t \ge 0}$ and $(\zeta^2_t)_{t \ge 0}$ with density~$p \in (0,1]$ and speed~$\mathsf{v}>0$ on~$\Z^d$, assume that both are started from stationarity~$\zeta^1_0, \zeta^2_0$, and let~$\mathcal{T}_1$ and~$\mathcal{T}_2$ be independent instruction manuals for random walks with range $R$ on $\Z^d$. Given $x,y \in \Z^d$, it is easy to see that
\begin{equation}\label{MC_separate_RW}
    \big(\mathbb{X}^{\mathcal U_1, \mathcal T_1}_t(\zeta^1_0,x),\; \mathbb{A}^{\mathcal U_1, \mathcal T_1}_t(\zeta^1_0,x), \; \mathbb{B}^{\mathcal U_1, \mathcal T_1}_t(\zeta^1_0,x), \; \mathbb{X}^{\mathcal U_2, \mathcal T_2}_t(\zeta^2_0,y),\; \mathbb{A}^{\mathcal U_2, \mathcal T_2}_t(\zeta^2_0,y), \; \mathbb{B}^{\mathcal U_2, \mathcal T_2}_t(\zeta^2_0,y) \big)_{t \ge 0}
\end{equation}
is a Markov chain with jump rate function $r_\mathrm{separate}$. See the definition below Remark~\ref{remark_mc_xab}. We write
\begin{equation}\label{eq_notation_as_above}
\begin{aligned}
    & (X_t,E^1_t) = \big(\mathbb{X}^{\mathcal U_1, \mathcal T_1}_t(\zeta^1_0,x),\; \mathbb{A}^{\mathcal U_1, \mathcal T_1}_t(\zeta^1_0,x) \cup \mathbb{B}^{\mathcal U_1, \mathcal T_1}_t(\zeta^1_0,x) \big), \\[0.1cm]
    & (Y_t,E^2_t) = \big(\mathbb{X}^{\mathcal U_2, \mathcal T_2}_t(\zeta^2_0,y),\; \mathbb{A}^{\mathcal U_2, \mathcal T_2}_t(\zeta^2_0,y) \cup \mathbb{B}^{\mathcal U_2, \mathcal T_2}_t(\zeta^2_0,y) \big), \\[0.1cm]
    & \F_t = \sigma \big( \mathcal{U}_1 \cap [0,t], \mathcal{U}_2 \cap [0,t], \mathcal{T}_1 \cap [0,t],\mathcal{T}_2 \cap [0,t] \big), \hspace{3mm} t \ge 0.
\end{aligned}
\end{equation}
We define the stopping times $\{\sigma_n\}_{n \in \N_0}$ by setting $\sigma_0 = 0$ and for~$n \geq 1$,
\begin{equation}\label{reg_times_def}
    \sigma_n = \min \{ m \in \N: m > \sigma_{n-1} \text{ and } E^1_m \cup E^2_m = \varnothing \}.
\end{equation}
For any $m \in \N$, let $J^1_m$ (resp. $J^2_m)$ denote the number of attempted steps (that is, arrivals in the instruction manual) by $X_t$ (resp. $Y_t$) during $[m-1,m]$,~$i=1,2$. We also write~$S^i_n := \sum^{\sigma_n}_{k= \sigma_{n-1}+1}J^i_m$ and~$S_n := S^1_n + S^2_n$ for any $n \in \N$. For any initial positions~$x,y \in \Z^d$ of the random walks we have the following properties:
\begin{enumerate}
    \item For every $n \in \N$, $(\zeta^1_{\sigma_n},\zeta^2_{\sigma_n})$ has distribution $\pi^\mathrm{edge}_p \otimes \pi^\mathrm{edge}_p$ and is independent of $( X_{\sigma_n},Y_{\sigma_n})$.
    \item The sequence $\{\sigma_{n+1} - \sigma_n\}_{n \geq 0}$ is i.i.d. Moreover, there exists~$a>0$ and~$\theta>0$ such that for every~$n \in \N$,
    \begin{equation}\label{ExpMomTaukBar1}
    \E \big[ \exp \big(\theta \cdot (\sigma_n - \sigma_{n-1}) \big) \mid \F_{\sigma_{n-1}} \big] \leq a \text{ a.s.}    
    \end{equation}
    and consequently,
    \begin{equation}\label{ExpMomTaukBar2}\E \big[ \exp \big(\theta \cdot \sigma_n \big) \big] \leq a^n.\end{equation}
    \item $\big( X_{\sigma_n},Y_{\sigma_n} \big)_{n \geq 0}$ is a random walk in $\Z^d\times \Z^d$ whose increments have the same distribution as~$\big(X_{\sigma_1}-x,Y_{\sigma_1}-y \big)$. 
    Moreover, there exists $b>0$ and $\beta>0$ such that for every $n \in \N$,
    \begin{equation}\label{ExpMomSkBar1}
    \E \big[ \exp \big(\beta \cdot S_n  \big) \mid \F_{\sigma_{n-1}} \big] \leq b \text{ a.s.}    
    \end{equation}
    and consequently,
\begin{equation}\label{ExpMomSkBar2}   
\E \big[ \exp \big(\beta \cdot (S_1 + \cdots + S_n ) \big) \big] \leq b^n.
\end{equation}
\end{enumerate}
Properties 1, 2 and 3 are proved with standard arguments, which have for instance been employed in Section 3 of \cite{PeresStaufferSteif2015}, requiring only minor modifications.

Now we prove that $f_\ell(x,y)$ and $g_\ell(x,y)$ from~\eqref{g_function_x,y} converge to 0 as $|x-y|_1 \to \infty$. After this, we prove Proposition \ref{Bigtheorem} and Corollary \ref{CoroBigThm}.
\begin{lemma}\label{lemma_limit_fl_g}
Assume $d \ge 3$. The following statements hold for any $\ell \ge 0$:
\begin{align}
    & 1. \displaystyle \lim_{L \to \infty} \sup_{x,y:|x-y|_1 > L} f_\ell(x,y) = 0; \hspace{9cm} \label{f_ell_limit_0}  \\ 
    & 2. \displaystyle \lim_{L \to \infty} \sup_{x,y:|x-y|_1 > L} g_\ell(x,y) = 0. \label{g_ell_limit_0}
\end{align}
\end{lemma}

\begin{proof}
1. Fix $\ell \ge 0$. Recall notation~\eqref{eq_notation_as_above}. Let
$$ I = \{t \ge 0:\min\big( |X_t-Y_t|_1, \hspace{0.5mm} \mathrm{dist}(X_t,E^2_t), \hspace{0.5mm} \mathrm{dist}(Y_t,E^1_t) \big) \leq \ell \} \; \text{ and } \; \tau = \inf{I}.$$
Let $m$ be the Lebesgue measure on $(0,\infty)$. We have that
\begin{equation}\label{g_ineq_lower_tau_infty}
    \begin{aligned}
    g_\ell(u,v) & \geq \E[ m(I) ] \geq \E[ m(I) \cdot \mathds{1}\{[\tau, \tau +1] \subseteq I \} ] \\[.2cm]
    & \geq \E[ \mathds{1}\{\tau < \infty\} \cdot \mathds{1}\{[\tau, \tau +1] \subseteq I \}] = \E[ \mathds{1}\{\tau < \infty\} \cdot \Prob([\tau, \tau +1] \subseteq I \mid \mathcal{F}_\tau) ] \\[.2cm]
    & \stackrel{(\ast)} \geq \Prob( \tau < \infty ) \exp(-2 \cdot \min(1,\mathsf{v})) = f_\ell(x,y)\exp(-2 \cdot \min(1,\mathsf{v})),
\end{aligned}
\end{equation}
where $(\ast)$ follows by observing that at time $\tau$, some vertex or edge from $X_t$ or $E^1_t$ is close to some vertex or edge from $Y_t$ or $E^2_t$; both of these encounter no update within one-time unit of $\tau$ with probability at least $\exp(-2 \cdot \min(1,\mathsf{v}))$. Hence, from~\eqref{g_ineq_lower_tau_infty}, we see that~\eqref{f_ell_limit_0} follows from~\eqref{g_ell_limit_0}. \\

\noindent 2. We have 
\begin{equation}\label{eq_3_limit}
    \begin{aligned}
        \E \Big[ \int^\infty_0 & \mathds{1}\{ \mathrm{dist}(X_t,E^2_t) \leq  2R \} \;\mathrm{d}t \Big]
        = \sum^\infty_{n = 0} \E \Big[ \int^{\sigma_{n+1}}_{\sigma_n} \mathds{1}\{ \mathrm{dist}(X_t,E^2_t) \leq 2R \} \;\mathrm{d}t \Big]  \\
        & \leq \sum^\infty_{n=0} \E \Big[ \int^{\sigma_{n+1}}_{\sigma_n} \mathds{1}\{ \mathrm{dist}(X_t,E^2_t) \leq 2R, |X_{\sigma_n} - Y_{\sigma_n}| > n^{1/8} \} \;\mathrm{d}t\Big]  \\ 
        & \hspace{3mm} + \sum^\infty_{n=0} \E \big[ (\sigma_{n+1}-\sigma_n) \cdot \mathds{1}\{ |X_{\sigma_n}-Y_{\sigma_n}|_1 \leq n^{1/8} \} \big].   
    \end{aligned}
\end{equation}
We want to prove that the sum on the right-hand side of the equality in \eqref{eq_3_limit} is summable. To achieve that, we show that the two sums on the right-hand side of the inequality in~\eqref{eq_3_limit} are finite. 

Note that if $ \mathrm{dist}(X_t,E^2_t) \leq 2R $ in the interval $[\sigma_n, \sigma_{n+1}]$ and $|X_{\sigma_n} - Y_{\sigma_n}| > n^{1/8}$, then we have that~$R \cdot S^1_{n+1} \ge n^{1/8} - 2R$. Hence,
\begin{equation}\label{eq_5_limit}
    \begin{aligned}
        \E \Big[ & \int^{\sigma_{n+1}}_{\sigma_n} \mathds{1} \{ \mathrm{dist}(X_t,E^2_t)  \leq 2R, \hspace{1mm} |X_{\sigma_n} - Y_{\sigma_n}| > n^{1/8} \} \mathrm{d}t \Big] \\ 
        & \leq \E[(\sigma_{n-1}-\sigma_n)\cdot \mathds{1} \{R \cdot S^1_{n+1} \geq n^{1/8}-2R \} ] \stackrel{(\ast)}{\leq} (\E [\sigma^2_1])^{1/2} \cdot \Prob( R\cdot S^1_{n+1} \geq n^{1/8}- 2R )^{1/2} \\
        & \hspace{-2mm} \stackrel{\eqref{ExpMomSkBar1}} \leq (\E [\sigma^2_1])^{1/2} \cdot b^{1/2}\exp \Big( - \frac{\beta(n^{1/8}-2R)}{2R} \Big),
    \end{aligned}
\end{equation}
where $(\ast)$ follows by the Cauchy-Schwarz inequality. Since $\E [\sigma^2_1])<\infty$~\eqref{ExpMomTaukBar2}, we have that this sequence is summable. We turn to the other sum. By the Strong Markov property,
\begin{equation}\label{eq_6_limit}
    \E \big[ (\sigma_{n+1}-\sigma_n) \cdot \mathds{1}\{ |X_{\sigma_n}-Y_{\sigma_n}|_1 \leq n^{1/8} \} \big] \stackrel{\eqref{ExpMomTaukBar1}} \leq \frac{2a}{\theta} \cdot \Prob( |X_{\sigma_n} - Y_{\sigma_n}|_1 \leq n^{1/8} ).
\end{equation}
Note that $(X_{\sigma_n} - Y_{\sigma_n})_{n \geq 0}$ is a random walk whose increments have the same distribution as 
\begin{equation}\label{increments_rw_U}
    U = X_{\sigma_1} - x - (Y_{\sigma_1}-y).
\end{equation}
This distribution generates an aperiodic, symmetric (hence with mean zero), irreducible walk supported on $\Z^d$. Moreover, $U$ has an exponential moment by \eqref{ExpMomSkBar2} and the inequality~$|U|_1 \leq S_1\cdot R$. By a bound derived from the Local Central Limit Theorem (see Proposition~2.4.4 of \cite{Lawler2010}), there is $c < \infty$ such that for all $n \in \N$ and $z \in \Z^d$,~$\Prob( X_{\sigma_n} - Y_{\sigma_n} = z) \leq cn^{-d/2}$. Hence, for any $n \ge 1$
\begin{equation}\label{MeetinginaBallIndependent}
    \begin{aligned}
    \Prob \big(  |X_{\sigma_n} - Y_{\sigma_n}|_1 \leq n^{1/8} \big) & = \hspace{-4mm} \sum_{w \in B_1(n^{1/8})} \hspace{-3mm} \Prob\big( X_{\sigma_n} - Y_{\sigma_n} = w-X_0 + Y_0\big) \leq \frac{c \cdot c'(d)}{m^{d/2 - d/8}}.   
    \end{aligned}
\end{equation}
Since $\tfrac{d}{2} - \tfrac{d}{8} > 1$ for every $d \geq 3$, the sequence \eqref{MeetinginaBallIndependent} is summable. Therefore, the sum on the right-hand side of the equality in \eqref{eq_3_limit} is summable. Given $\epsilon>0$, there exists $N \in \N$ such that
$$ \E \Big[ \int^\infty_0 \mathds{1}\{ \mathrm{dist}(X_t,E^2_t) \leq  2R \} \;\mathrm{d}t \Big]
        \leq  \E \Big[ \int^{\sigma_N}_0 \mathds{1}\{ \mathrm{dist}(X_t,E^2_t) \leq 2R \} \;\mathrm{d}t \Big] + \epsilon/2. $$
Note that
\begin{equation*}
    \E \Big[ \int^{\sigma_{N}}_0 \mathds{1}\{ \mathrm{dist}(X_t,E^2_t) \leq 2R \} \mathrm{d}t \Big]  \leq \E \Big[ \sigma_N \cdot \mathds{1} \Big\{ S_1 + \dots + S_N \geq \frac{|x-y|_1 - 2R}{R} \Big\} \Big]. 
\end{equation*}
Following the same procedure done in~\eqref{eq_5_limit}, we see that the integral on the left-hand side can now be made smaller than $\epsilon/2$ by taking $L$ large enough. We have proved that 
$$\lim_{L \to \infty}\sup_{x,y:|x-y|_1 > L}\E \Big[ \int^\infty_0 \mathds{1} \{ \mathrm{dist}(X_t,E^2_t) \leq 2R \} \mathrm{d}t \Big] = 0.$$ 
An analogous procedure gives
$$\lim_{L \to \infty}\sup_{x,y:|x-y|_1 > L} \Big( \E \Big[ \int^\infty_0 \mathds{1}\{ \mathrm{dist}(Y_t,E^1_t) \leq 2R \} \mathrm{d}t \Big] + \E \Big[ \int^\infty_0 \mathds{1}\{ |X_t - Y_t|_1 \leq 2R \} \Big] \Big) = 0.$$ 
\end{proof}

\begin{proof}[Proof of Proposition \ref{Bigtheorem}]
Fix $\ell \ge 0$. We assume that the process is built from a graphical construction: we let~$\mathcal U$ encode the environment updates, assume that $\zeta_0 \sim \pi^\mathrm{edge}_p$, and~$\mathcal T_1$ and~$\mathcal T_2$ be the instruction manuals for the walkers started at~$x$ and~$y$. 
Lemma~\ref{lemma_main} and Lemma~\ref{lemma_limit_fl_g} imply that
\begin{equation}\label{NoMeetingxyFarStationarity}
    \lim_{L \to \infty}\sup_{|x-y|_1 > L}\Prob( \exists t \geq 0 : |\mathbb{X}^{\mathcal U, \mathcal T_1} _t(\zeta_0,x) - \mathbb{X}^{\mathcal U, \mathcal T_2}_t(\zeta_0,y)|_1 \leq \ell ) = 0.  
\end{equation}
Now we extend (\ref{NoMeetingxyFarStationarity}) to any configuration $\zeta$. For any~$E \in \mathcal{P}_{\mathrm{fin}}( E(\Z^d) )$, we consider~$\zeta^{\zeta \to E}_0$ as in~\eqref{XiEzeta}. Let $E(\zeta) = \{ e \in E: \zeta(e) = 1 \}$, then for any $x,y \in \Z^d$,
\begin{equation}\label{sum_1_collision}
    \begin{aligned}
    \Prob \big( \exists t & \geq 0:  |\mathbb{X}^{\mathcal U, \mathcal T_1}_t(\zeta^{\zeta \to E}_0,x) - \mathbb{X}^{\mathcal U, \mathcal T_2}_t(\zeta^{\zeta \to E}_0,y)|_1 \leq \ell \big) \\[.2cm]
    & = \Prob \big( \exists t \geq 0:  |\mathbb{X}^{\mathcal U, \mathcal T_1}_t(\zeta_0,x) - \mathbb{X}^{\mathcal U, \mathcal T_2}_t(\zeta_0,y)|_1 \leq \ell \mid \zeta_0 \equiv \zeta \text{ on } E \big) \\[.2cm]
    & \leq (p^{|E(\zeta)|}(1-p)^{|E|-|E(\zeta)|})^{-1} \cdot \Prob \big( \exists t \geq 0:  |\mathbb{X}^{\mathcal U, \mathcal T_1}_t(\zeta_0,x) - \mathbb{X}^{\mathcal U, \mathcal T_2}_t(\zeta_0,y)|_1 \leq \ell \big) \\[.2cm]
    & \leq c(p)^{|E|} \cdot \Prob \big( \exists t \geq 0:  |\mathbb{X}^{\mathcal U, \mathcal T_1}_t(\zeta_0,x) - \mathbb{X}^{\mathcal U, \mathcal T_2}_t(\zeta_0,y)|_1 \leq \ell \big),
    \end{aligned}
\end{equation}
where $c(p) = (\min\{p,1-p\})^{-1}$. Fix $\epsilon>0$ and choose~$r$ corresponding to~$\epsilon$ in Lemma \ref{ArgumentCouplingRW}. We then apply (\ref{sum_1_collision}) holds with~$E=E_r:=E_1(x,r)\cup E_1(y,r)$. Hence,
\begin{align*}
    \Prob ( \exists t \geq 0: |\mathbb{X}^{\mathcal U, \mathcal T_1}_t & (\zeta,x)  - \mathbb{X}^{\mathcal U, \mathcal T_2}_t(\zeta,y)|_1 \leq \ell \big) \\[.2cm]
    & \leq \Prob \big( \exists t \geq 0:  |\mathbb{X}^{\mathcal U, \mathcal T_1}_t(\zeta^{\zeta \to E_r}_0,x) - \mathbb{X}^{\mathcal U, \mathcal T_2}_t(\zeta^{\zeta \to E_r}_0,y)|_1 \leq \ell \big) + \epsilon \\[.2cm]
    & \leq c(p)^{|E_r|} \cdot \Prob \big( \exists t \geq 0:  |\mathbb{X}^{\mathcal U, \mathcal T_1}_t(\zeta_0,x) - \mathbb{X}^{\mathcal U, \mathcal T_2}_t(\zeta_0,y)|_1 \leq \ell \big) + \epsilon \stackrel{(\ast)} < 2\epsilon,
\end{align*}
where $(\ast)$ follows by (\ref{NoMeetingxyFarStationarity}) for $L$ large enough. 

\end{proof}
\begin{proof}[Proof of Corollary \ref{CoroBigThm}]
Fix $\ell \geq 0$. We assume that the process is built from a graphical construction: we let~$\mathcal U$ encode the environment updates and~$\mathcal T_1$ and~$\mathcal T_2$ be the instruction manuals for the random walks. Given~$\epsilon>0$, due to Proposition \ref{Bigtheorem} there exists~$L>0$ such that for~$|x-y|_1 \geq L$ and any~$\zeta \in \Omega_\mathrm{edge}$,
\begin{equation}\label{meeting_ell_cor}
    h(x,y,\zeta) := \Prob( \exists t \geq 0: |\mathbb{X}^{\mathcal U, \mathcal T_1}_t(\zeta,x) - \mathbb{X}^{\mathcal U, \mathcal T_2}_t(\zeta,y)|_1 \leq \ell ) < \epsilon/2.     
\end{equation}
Fix an arbitrary $\zeta_0 \in \Omega_\mathrm{edge}$. For $|x-y|_1 \geq L$, we have that for any $t \ge 0$,
$$ \Prob( \exists s \geq t: |\mathbb{X}^{\mathcal U, \mathcal T_1}_s(\zeta_0,x)-\mathbb{X}^{\mathcal U, \mathcal T_2}_s(\zeta_0,y)| \leq \ell ) \leq h(x,y,\zeta_0) < \epsilon. $$
For $|x-y|_1 < L$, we consider the stopping time 
\begin{equation}\label{def_tau_apart_L}
    \tau = \inf \big\{ m \in \N: |\mathbb{X}^{\mathcal U, \mathcal T_1}_m(\zeta_0,x) - \mathbb{X}^{\mathcal U, \mathcal T_2}_m(\zeta_0,y)|_1 \geq L \big\}.
\end{equation}
It is easy to see that there exists $b > 0$ such that 
$$ \inf_{ \substack{ \zeta \in \Omega_\mathrm{edge} \\ x,y: |x-y|_1 < L} } \Prob \big(|\mathbb{X}^{\mathcal U, \mathcal T_1}_1(\zeta,x) - \mathbb{X}^{\mathcal U, \mathcal T_2}_1(\zeta,y)|_1 \geq L \big) \geq b.$$
By the Markov property, $\Prob(\tau \geq n) \leq (1-b)^n$ and then $\tau < \infty$ a.s. Hence,
\begin{equation*}
    \begin{aligned}
    \Prob \big( \exists s \ge t: |\mathbb{X}^{\mathcal U, \mathcal T_1}_s & (\zeta_0,x)  - \mathbb{X}^{\mathcal U, \mathcal T_2}_s(\zeta_0,y)|_1  \leq \ell \big) \\[.2cm]
    & \leq \Prob \big( \exists s \geq \tau: |\mathbb{X}^{\mathcal U, \mathcal T_1}_s(\zeta_0,x) - \mathbb{X}^{\mathcal U, \mathcal T_2}_s(\zeta_0,y)|_1  \leq \ell \big) + \Prob( \tau > t ) \\[.2cm]
    & \hspace{-0.4mm} \stackrel{(\ast)} \leq \E \big[ h(\mathbb{X}^{\mathcal U, \mathcal T_1}_\tau(\zeta_0,x),\mathbb{X}^{\mathcal U, \mathcal T_2}_\tau(\zeta_0,y),\zeta_\tau) \big] + \Prob( \tau > t )  \\[.2cm]
    & \hspace{-2.4mm} \stackrel{\eqref{meeting_ell_cor}} \leq \epsilon/2 + (1-b)^{ \lfloor t \rfloor } \stackrel{(\ast \ast)}{<} \epsilon, \\
    \end{aligned}
\end{equation*}
where $(\ast)$ holds by the Strong Markov property and $(\ast \ast)$ by taking $t$ large enough. 
\end{proof}

\subsection{Duality and stationary measures}\label{duality_stat_measures_vmodyn_section}
In this section, we define a \textbf{system of coalescing random walks on dynamical percolation} and prove that this process is a dual for the voter model on dynamical percolation~$M_t = (\zeta_t,\eta_t)$. As a consequence, we obtain a family of stationary measures that will coincide with the set of extremal stationary measures for $(M_t)$.

Given $k \in \N$, recall notation $z^{x \to y}$ from~\eqref{z_x_arrow_y}, for~$\mathbf{y} = (y_1,\dots,y_k) \in (\Z^d)^k$,~$x,y \in \Z^d$, we write
$$\mathbf{y}^{x \to y} = (y^{x \to y}_1,\dots, y^{x \to y}_k).$$
Fix $R,d\in \N$, $p \in [0,1], \mathsf{v}>0$, a system of~$k$ \textbf{coalescing random walks on a single dynamical percolation environment} is a Markov process~$(Y^1_t,\ldots, Y^k_t, \zeta_t)_{t \ge 0}$ on~$(\mathbb Z^d)^k \times \Omega_\mathrm{edge}$ with generator \begin{align*}
    &\mathcal L_{\mathrm{crwdyn}}f(\mathbf{y},\zeta) = \mathcal L_{\mathrm{dp}}f_\mathbf{y}(\zeta) \\
    &+ \frac{1}{|B_1(R)|-1}\sum_{x \in \Z^d} \sum_{y \in B_1(x,R)} \mathds{1}{\{x \xleftrightarrow{\zeta } y \text{ inside }B_1(x,R)\}} \cdot \big(f_\zeta(\mathbf{y}^{x \to y}) - f_\zeta(\mathbf{y}) \big)
\end{align*}
where $f: (\Z^d)^k \times \Omega_\mathrm{edge} \to \R$ is any function such that $f_\mathbf{y}(\zeta)$ only depends on finitely many coordinates and $\zeta \in \Omega_\mathrm{edge}$.

\begin{remark}
Representing the positions of coalescing random walkers as an ordered list~$(Y^1_t,\ldots,Y^k_t)$, rather than a \textbf{set}, may seem odd; after all, if two coordinates of the ordered list are equal at some point in time, they stay equal afterwards. Still, it will be good to keep track of the locations of each of the walkers in a labelled manner, for some of our later purposes: notably, for a coupling with independent random walkers. For other purposes, we will be happy with discarding the labels and using only the set of locations; this is done for instance in~\eqref{coalescing_jump_rate} below.    
\end{remark}

In close analogy to the Markov chain \eqref{Markov_chain_same_environment}, we introduce a Markov chain~$(\mathcal C_t,\mathcal A_t,\mathcal B_t)_{t \ge 0}$, where $\mathcal{C}_t \in \mathcal{P}_\mathrm{fin}(\Z^d)$ corresponds to the set of locations of the coalescing random walks at time~$t$, and $\mathcal{A}_t$,~$\mathcal{B}_t$ in~$\mathcal P_{\mathrm{fin}}(E(\mathbb Z^d))$ are the sets of (open, closed, respectively) inspected edges by the coalescing random walks. In contrast to~\eqref{Markov_chain_same_environment}, where the two sets of inspected edges always started empty at time zero, here we will allow them to start from some initial knowledge. In other words, we now allow for~$\mathcal A_0$ and~$\mathcal B_0$ to be non-empty (but they must be disjoint). To write its jump rates, we introduce some notation. 

For~$x \in \Z^d, E,F \in \mathcal{P}_{\mathrm{fin}}(E(\Z^d))$ with $E \cap F = \varnothing$, we consider the set
$$ \mathcal{E}^R_x(E,F) := E_1(x,R) \setminus (E \cup F), $$
consisting of all edges in~$E_1(x,R)$ which are not in~$E$ or~$F$. We let
$$ \mathcal{B}^R_x(E,F) := \{ (E',F'): E',F' \subseteq E_1(x,R), E' \cap F' = \varnothing, E' \cup F' = \mathcal{E}^R_x(E,F) \},$$
that is,~$\mathcal B^R_x(E,F)$ is the collection of pairs~$(E',F')$ that give a partition~$\mathcal E^R_x(E,F)$. The sets~$E$ and~$E'$ will represent open edges and the sets~$F$ and~$F'$, closed edges.
When we fix a pair~$(E',F')\in \mathcal B^R_x(E,F)$, we can ask whether there is a path from~$x$ to~$y$ that stays inside~$B_1(x,R)$ and uses only edges in~$E \cup E'$. With this in mind, we define
$$ \mathcal{B}^R_{x \to y}(E,F) := \{ (E',F') \in \mathcal{B}^R_x(E,F): x \xleftrightarrow{\zeta_{E \cup E'}} y \text{ inside } B_1(x,R) \},$$
where $\zeta_{E \cup E'}(e) = 1$ if and only if $e \in E \cup E'$. We also define
$$ \mathcal{B}^R_{x \not \to y}(E,F):= \mathcal{B}^R_x(E,F) \setminus \mathcal{B}^R_{x \not \to y}(E,F).$$
Given $C \in \mathcal{P}_\mathrm{fin}(\Z^d)$, $x,y \in \Z^d$, we write~$C^y_x := (C \setminus \{x\}) \cup \{y\}$. The jump rates of the Markov chain $(\mathcal C_t, \mathcal A_t, \mathcal B_t)$ are:
\begin{equation}\label{coalescing_jump_rate}
    \begin{array}{ll}
        \vspace{1.5 mm} r\big( (C,E,F);(C,E \setminus \{e\}, F \setminus \{e\} ) \big) = \mathsf{v}, & e \in E \cup F  \\[.2cm]
        r\big( (C,E,F);( C_x^y,E \cup E', F \cup F' ) \big) = \frac{ p^{|E'|} (1-p)^{|F'|} }{|B_1(R)|-1}, & x \in C, y \in B_1(x,R)\setminus\{x\}, \\
        & (E',F') \in \mathcal{B}^R_{x \to y}(E,F). \\[.2cm]
        r\big( (C,E,F);( C,E \cup E', F \cup F' ) \big) = \frac{ p^{|E'|} (1-p)^{|F'|} }{|B_1(R)|-1}, & x \in C, y \in B_1(x,R)\setminus\{x\}, \\
        & (E',F') \in \mathcal{B}^R_{x \not \to y}(E,F), \\
        \end{array}
\end{equation}   

In Section~\ref{section_characterization_stat_measures}, we give a constructive definition of $(\mathcal C_t,\mathcal A_t, \mathcal B_t)$ using a graphical construction for the environment and instruction manuals for the random walks (so we identify the coalescing particles). This construction is made to obtain a coupling of a system of coalescing random walks with a system of independent random walks on a single environment. 

To state the duality relationship, we introduce some notation. Given $C \in \mathcal{P}_\mathrm{fin}(\Z^d)$ and~$E,F \in \mathcal{P}_\mathrm{fin}( E(\Z^d) )$ with $E \cap F = \varnothing$, we denote by $\mathbf{P}_{C,E,F}$ the probability measure under which the Markov chain $(\mathcal C_t,\mathcal A_t,\mathcal B_t)$ is defined and satisfies $\mathbf{P}_{C,E,F}(\mathcal C_0 = C,\mathcal A_0=E,\mathcal B_0=F) = 1$.

Given $(\eta,\zeta) \in \Omega_\mathrm{site} \times \Omega_\mathrm{edge}$, we denote by $P_{\eta,\zeta}$ a probability measure under which the voter model on dynamical percolation~$(M_t)_{t \geq 0} = (\eta_t,\zeta_t)_{t \ge 0}$ is defined and satisfies~$P_{\eta,\zeta}\big( M_0 = (\eta,\zeta) \big) =1$. Given a probability distribution~$\mu$ on~$\Omega_\mathrm{site} \times \Omega_\mathrm{edge}$, we write~$P_\mu = \int P_{\eta,\zeta} \; \mu(\mathrm{d}(\eta,\zeta))$. Also, we write $P_{\eta,\pi^\mathrm{edge}_p} = \int P_{\eta,\zeta}\; \pi^\mathrm{edge}_p(\mathrm{d} \zeta)$.

\begin{lemma}[Duality]\label{Duality}
The Markov processes $(M_t)_{t \geq 0}$ and $(\mathcal C_t,\mathcal A_t,\mathcal B_t)_{t \geq 0}$ are dual with respect to 
$$H(\eta,\zeta,C,E,F) = p^{-|E|} (1-p)^{-|F|} \cdot \phi(C,\eta) \cdot \varphi(E,F,\zeta),$$
where $\phi(C,\eta) = \prod_{x \in C}\eta(x)$ and $\varphi(E,F,\zeta) = \prod_{e \in E}\zeta(e) \cdot \prod_{e \in F}(1-\zeta(e))$. Then,
\begin{equation}\label{dual_equation_H}
    E_{\eta,\zeta}[ H(\eta_t, \zeta_t, C, E, F)] = \mathbf{E}_{C,E,F}[H( \eta, \zeta, \mathcal C_t, \mathcal A_t, \mathcal B_t )].
\end{equation}
\end{lemma}

\begin{proof}[Proof of Lemma \ref{Duality}]
We write
$$ H_{\eta,\zeta}(C,E,F) = H_{C,E,F}(\eta,\zeta) = H(\eta,\zeta,C,E,F).$$
We aim to use Theorem 3.42 of \cite{Liggett2010}, which states that a sufficient condition for duality is to prove that 
$$ \mathcal{L}_\mathrm{vmdyn}H_{C,E,F}(\eta,\zeta) = \sum_{\Bar{C},\Bar{E},\Bar{F}}r\big( (C,E,F);(\Bar{C},\Bar{E},\Bar{F}) \big)\cdot \big[ H_{\eta,\zeta}(\Bar{C},\Bar{E},\Bar{F}) - H_{\eta,\zeta}(C,E,F) \big]. $$
We split this sum on the right-hand side of the last equation into the dynamical percolation and the random walks parts. The former one:
$$\sum_{\Bar{E},\Bar{F}} r\big( (C,E,F);( C, \Bar{E}, \Bar{F}) \big) \cdot \big[ H_{\eta,\zeta}(C,\Bar{E},\Bar{F}) - H_{\eta,\zeta}(C,E,F) \big].$$
For $e \in F$,
\begin{align*}
    r \big( (C &,E,F);( C, E \setminus \{e\}, F \setminus \{e\}) \big) \cdot \big[ H_{\eta,\zeta}(C,E \setminus \{e\},F \setminus \{e\}) - H_{\eta,\zeta}(C,E,F) \big] \\[.2cm]
    & = \mathsf{v} \cdot \phi(C,\eta) \cdot \big[ (p^{|E|} (1-p)^{|F|-1})^{-1} \cdot \varphi(E,F\setminus \{e\},\zeta) - (p^{|E|} (1-p)^{|F|})^{-1} \cdot \varphi(E,F,\zeta) \big] \\[.2cm]
    & = \mathsf{v} \cdot (p^{|E|}(1-p)^{|F|})^{-1} \cdot \phi(C,\eta) \cdot \big[ (1-p) \cdot \varphi(E,F \setminus \{e\},\zeta) - \varphi(E,F,\zeta) \big] \\[.2cm]
    & = (p^{|E|}(1-p)^{|F|})^{-1} \cdot \phi(C,\eta) \cdot c(e,\zeta) \cdot \big[ \varphi(E,F,\zeta_e) - \varphi(E,F,\zeta)  \big] \\[.2cm]
    & = c(e,\zeta) \cdot [ H_{C,E,F}(\eta,\zeta_e) - H_{C,E,F}(\eta,\zeta) ].
\end{align*}
An analogous computation applies to $e \in E$. Summing over $e \in E \cup F$, we have
$$\sum_{e \in E(\Z^d)}c(e,\zeta) \cdot \big[ H_{C,E,F}(\eta,\zeta_e) - H_{C,E,F}(\eta,\zeta) \big]. $$
We now turn to the random walk part. We drop the indices $(E,F)$ and $R$ in $\mathcal{B}^R_{x \to y}(E,F)$ and $\mathcal{B}^R_{x \not \to y}(E,F)$.   
\begin{align*}
    & \sum_{\Bar{C},\Bar{E},\Bar{F}} r\big( (C,E,F);(\Bar{C}, \Bar{E}, \Bar{F}) \big) \cdot \big[ H_{\eta,\zeta}(\Bar{C},\Bar{E},\Bar{F}) - H_{\eta,\zeta}(C,E,F) \big] \\
    & = \sum_{\substack{x,y \in \Z^d: \\ 0< |x-y|_1 \leq R}} \frac{ 1 }{|B_1(R)|-1} \Bigg( \sum_{ (E',F') \in \mathcal{B}_{x \to y}} p^{|E'|}(1-p)^{|F'|} \cdot \big[ H_{\eta,\zeta}(C^y_x,E \cup E',F \cup F') - H_{\eta,\zeta}(C,E,F) \big] \\
    & + \sum_{ (E',F') \in \mathcal{B}_{x \not \to y}} p^{|E'|}(1-p)^{|F'|} \cdot \big[ H_{\eta,\zeta}(C,E \cup E',F \cup F') - H_{\eta,\zeta}(C,E,F) \big] \Bigg).
\end{align*}
We fix $x \in C$ and $y \in B_1(x,R)$ and assume that $\zeta \equiv 1$ on $E$ and $\zeta \equiv 0$ on $F$, i.e., $\varphi(E,F,\zeta)=1$. For $(E',F') \in \mathcal{B}_{x \to y}$, 
\begin{equation}\label{VMSum}
    \begin{aligned}
    & H_{\eta,\zeta}(C^y_x,E \cup E',F \cup F') - H_{\eta,\zeta}(C,E,F) \\[.2cm]
    & =  (p^{|E|+|E'|}(1-p)^{|F|+|F'|})^{-1} \cdot \phi(C^y_x,\eta)\cdot \varphi(E',F',\zeta) -(p^{|E|}(1-p)^{|F|})^{-1} \cdot \phi(C,\eta)\\[.2cm]
    & = (p^{|E|}(1-p)^{|F|})^{-1} \cdot \big[ (p^{|E'|}(1-p)^{|F'|})^{-1} \cdot \phi(C,\eta^{y \to x})\cdot \varphi(E',F',\zeta) - \phi(C,\eta) \big].    
    \end{aligned}
\end{equation}
Note that under the assumption $\zeta \equiv 1$ on $E$ and $\zeta \equiv 0$ on $F$,
$$ \sum_{(E',F') \in \mathcal{B}_{x \to y}} \varphi(E',F',\zeta) = \mathds{1} \{ x \stackrel{\zeta}{\leftrightarrow} y \text{ inside } B(x,R) \}. $$
We write $r_{x \to y} := \sum_{ (E',F') \in \mathcal{B}_{x \to y}}p^{|E'|}(1-p)^{|F'|}$. Multiplying (\ref{VMSum}) by $p^{|E'|}(1-p)^{|F'|}$ and summing over $(E',F') \in \mathcal{B}_{x \to y}$ we have
\begin{equation}\label{duality_sum_1}
    \begin{aligned}
    & \sum_{ (E',F') \in \mathcal{B}_{x \to y}} p^{|E'|}(1-p)^{|F'|} \cdot \big[ H_{\eta,\zeta}(\Bar{C},E \cup \Bar{E},F \cup \Bar{F}) - H_{\eta,\zeta}(C,E,F) \big] \\[.2cm]
    & = \sum_{ (E',F') \in \mathcal{B}_{x \to y}} (p^{|E|}(1-p)^{|F|})^{-1} \cdot \big[ \phi(C,\eta^{y \to x})\cdot \varphi(E',F',\zeta) - (p^{|E'|}(1-p)^{|F'|}) \cdot \phi(C,\eta) \big] \\[.2cm]
    & = (p^{|E|}(1-p)^{|F|})^{-1} \cdot \big[ \phi(C,\eta^{y \to x})\cdot \mathds{1}\{ x \stackrel{\zeta}{\leftrightarrow} y \text{ inside } B(x,R) \} - r_{x \to y} \cdot \phi(C,\eta) \big].
\end{aligned}    
\end{equation}
We do an analogous computation for $(E',F') \in \mathcal{B}_{x \not \to y}$,
\begin{equation*}
    \begin{aligned}
    & H_{\eta,\zeta}(C,E \cup E',F \cup F') - H_{\eta,\zeta}(C,E,F) \\[.2cm]
    & = (p^{|E|}(1-p)^{|F|})^{-1} \cdot \big[ (p^{|E'|}(1-p)^{|F'|})^{-1} \cdot \phi(C,\eta)\cdot \varphi(E',F',\zeta) - \phi(C,\eta) \big].    
    \end{aligned}
\end{equation*}
We write $r_{x \not \to y } := \sum_{(E',F') \in \mathcal{B}_{x \not \to y}}p^{|E'|}(1-p)^{|F'|}$. Then
\begin{equation}\label{duality_sum_2}
\begin{aligned}
    & \sum_{ (E',F') \in \mathcal{B}_{x \not \to y}} (p^{|E'|}(1-p)^{|F'|}) \cdot \big[  H_{\eta,\zeta}(C,E \cup E',F \cup F') - H_{\eta,\zeta}(C,E,F) \big] \\
    & = \sum_{ (E',F') \in \mathcal{B}_{x \not \to y}} (p^{|E|}(1-p)^{|F|})^{-1} \cdot \big[ \phi(C,\eta)\cdot \varphi(E',F',\zeta) - (p^{|E'|}(1-p)^{|F'|}) \cdot \phi(C,\eta) \big] \\
    & = (p^{|E|}(1-p)^{|F|})^{-1} \cdot \big[ \phi(C,\eta)\cdot (1-\mathds{1} \{ x \stackrel{\zeta}{\leftrightarrow} y \text{ inside } B(x,R) \}) - r_{x \not \to y} \cdot \phi(C,\eta) \big].
\end{aligned}    
\end{equation}
Recall that we are assuming that $\varphi(E,F,\zeta) = 1$ and note that $r_{x \to y } + r_{x \not \to y } = 1$. Hence, summing (\ref{duality_sum_1}) and (\ref{duality_sum_2}),
\begin{align*}
    \sum_{\Bar{A},\Bar{E},\Bar{F}} & r\big( (C,E,F);(\Bar{C}, \Bar{E}, \Bar{F}) \big) \cdot \big[ H_{\eta,\zeta}(\Bar{C},\Bar{E},\Bar{F}) - H_{\eta,\zeta}(C,E,F) \big] \\
    & = \sum_{\substack{x,y \in \Z^d: \\ 0< |x-y|_1 \leq R}} \frac{ (p^{|E|}(1-p)^{|F|})^{-1} }{|B_1(R)|-1} \cdot \Big( \phi(C,\eta^{y \to x}) - \phi(C,\eta) \Big) \cdot \mathds{1} \{ x \stackrel{\zeta}{\leftrightarrow} y \text{ inside } B(x,R) \} \\
    & = \sum_{\substack{x,y \in \Z^d: \\ 0< |x-y|_1 \leq R}} \frac{ 1 }{|B_1(R)|-1} \cdot \Big[ H_{C,E,F}(\eta^{y \to x},\zeta) - H_{C,E,F}(\eta,\zeta) \Big].
\end{align*}
Finally, we have the duality relationship by summing the two parts (random walks and dynamical percolation). 
\end{proof}
An immediate consequence of the duality relationship is the following lemma which will be crucial to characterise the stationary measures of the voter model on dynamical percolation.

\begin{lemma}\label{DualityNice}
Let $(\eta,\zeta) \in \Omega_\mathrm{site} \times \Omega_\mathrm{edge}$ and~$C \in \mathcal{P}_{\mathrm{fin}}(\Z^d)$,~$E,F \in \mathcal{P}_{\mathrm{fin}}(E(\Z^d))$ with~$E \cap F = \varnothing$, we have that for any $t \ge 0$,
\begin{equation}\label{DualityNice_equation}
    \begin{aligned}
        P_{\eta,\zeta}(\eta_t \equiv 1 \text{ on } & C, \zeta_t \equiv 1 \text{ on } E, \zeta_t \equiv 0 \text{ on } F ) \\
        & = p^{|E|}(1-p)^{|F|} \cdot \mathbf{E}_{C,E,F} \big[ p^{-|\mathcal A_t|}(1-p)^{-|\mathcal B_t|} \cdot \phi(\mathcal C_t,\eta) \cdot \varphi(\mathcal A_t, \mathcal B_t,\zeta) \big].
    \end{aligned}
\end{equation}
\end{lemma}
\begin{proof}
It follows from the duality relationship that 
\begin{equation*}
    \begin{aligned}
        P_{\eta,\zeta}(\eta_t \equiv 1 \text{ on } C &,\; \zeta_t \equiv 1 \text{ on } E ,\; \zeta_t \equiv 0 \text{ on } F ) \\[.2cm]
        & = p^{|E|}(1-p)^{|F|} \cdot E_{\eta,\zeta}H( \eta_t,\zeta_t,C,E,F ) \\[.2cm]
        & \hspace{-2.2mm} \stackrel{\eqref{dual_equation_H}} = p^{|E|}(1-p)^{|F|} \cdot \mathbf{E}_{C,E,F}H( \eta,\zeta,\mathcal C_t, \mathcal A_t, \mathcal B_t ) \\[.2cm]
        & = p^{|E|}(1-p)^{|F|} \cdot \mathbf{E}_{C,E,F} \big[ p^{-|\mathcal A_t|}(1-p)^{-|\mathcal B_t|} \cdot \phi(\mathcal C_t,\eta) \cdot \varphi(\mathcal A_t, \mathcal B_t,\zeta) \big]. 
    \end{aligned}    
\end{equation*}    
\end{proof}
From Lemma~\ref{DualityNice}, we derive two corollaries. First, we prove the existence of the stationary measures $\mu^\mathrm{dyn}_\alpha$, $\alpha \in [0,1]$, of Theorem \ref{ThmCharacterization} (characterization of stationary measures are proved in Section \ref{section_characterization_stat_measures}). The second is Corollary~\ref{CoroDuality}, which will be applied in Section \ref{section_characterization_stat_measures}. Before proceeding, we introduce some notation. 

Recall that $\mathcal{I}$ represents the set of all stationary measures of the voter model on dynamical percolation. Let $\mu$ be a measure on $\Omega_\mathrm{site} \times \Omega_\mathrm{edge}$, we write
\begin{equation}\label{mu_bar_definition}
    \begin{aligned}
    &\widehat{\mu}(C,E,F) = \mu( (\eta,\zeta): \eta \equiv 1 \text{ on } C,\; \zeta \equiv 1 \text{ on } E,\; \zeta \equiv 0 \text{ on } F ) \quad \text{and} \\
    &\Bar{\mu}(C,E,F) = \mu( (\eta,\zeta): \eta \equiv 1 \text{ on } C \mid \zeta \equiv 1 \text{ on } E,\; \zeta \equiv 0 \text{ on } F ).    
    \end{aligned}
\end{equation}

\begin{proof}[Proof of Theorem \ref{ThmCharacterization}]
Integrating~\eqref{DualityNice_equation} over $\pi^\mathrm{edge}_p(\mathrm{d}\zeta)$ and applying Fubini's theorem we have that
\begin{equation}\label{DualityPi}
    P_{\eta,\pi^\mathrm{edge}_p}(\eta_t \equiv 1 \text{ on } C,\; \zeta_t \equiv 1 \text{ on } E,\; \zeta_t \equiv 0 \text{ on } F ) = p^{|E|}(1-p)^{|F|} \cdot \mathbf{P}_{C,E,F} ( \eta \equiv 1 \text{ on } \mathcal C_t ).  
\end{equation}
Recall the Bernoulli product measure $\pi^\mathrm{site}_\alpha$ on $\Z^d$. Then,
\begin{equation*}
    \begin{aligned}
    P_{\pi^\mathrm{site}_\alpha \otimes \pi^\mathrm{edge}_p}\big( \eta_t \equiv 1 & \text{ on } C,\;\zeta_t \equiv 1 \text{ on } E,\; \zeta_t \equiv 0 \text{ on } F \big) \\[.2cm]
    & = \int P_{\eta \otimes \pi^\mathrm{edge}_p}\big( \eta_t \equiv 1 \text{ on } C,\;\zeta_t \equiv 1 \text{ on } E,\;\zeta_t \equiv 0 \text{ on } F \big) \;\pi^\mathrm{site}_\alpha(d\eta) \\[.2cm]
    & \hspace{-2.3mm} \stackrel{(\ref{DualityPi})}{=} p^{|E|}(1-p)^{|F|} \cdot \int \mathbf{P}_{C,E,F} ( \eta \equiv 1 \text{ on } \mathcal C_t) \;\pi^\mathrm{site}_\alpha(d\eta) \\[.2cm]
    & = p^{|E|}(1-p)^{|F|} \cdot \mathbf{E}_{C,E,F} \big[ \alpha^{|\mathcal C_t|} \big].
    \end{aligned}
\end{equation*}
Let us denote $|\mathcal C_\infty| = \lim_{t \to \infty}|\mathcal C_t|$; the limit exists because $|\mathcal C_t|$ decreases with $t$. By taking the limit on the last equality as $t \to \infty$, we have that
$$\lim_{t \to \infty}P_{\pi^\mathrm{site}_\alpha \otimes \pi^\mathrm{edge}_p}\big( \eta_t \equiv 1 \text{ on } C,\zeta_t \equiv 1 \text{ on } E, \zeta_t \equiv 1 \text{ on } F \big) =  p^{|E|}(1-p)^{|F|} \cdot \mathbf{E}_{C, E,F} \big[ \alpha^{|\mathcal C_\infty|} \big]. $$
Therefore, under $P_{\pi^\mathrm{site}_\alpha \otimes \pi^\mathrm{edge}_p}$, as $t \to \infty$, $(\eta_t,\zeta_t)$ converges in distribution to a measure~$\mu^{\mathrm{dyn}}_\alpha$ on~$\Omega_\mathrm{site}\times\Omega_\mathrm{edge}$ characterized by
\begin{equation}\label{mu_alpha_def_with_A_infty}
    \bar \mu^{\mathrm{dyn}}_\alpha( C, E, F ) \stackrel{\eqref{mu_bar_definition}} = p^{-|E|}(1-p)^{-|F|} \cdot \widehat \mu^\mathrm{dyn}_\alpha(C,E,F) =  \mathbf{E}_{C,E,F} \big[ \alpha^{|\mathcal C_\infty|} \big].
\end{equation}
As the measures $\mu^{\mathrm{dyn}}_\alpha$ are obtained as distributional limits of $(\eta_t,\zeta_t)$, they are stationary for this process.    
\end{proof}

\begin{corollary}\label{CoroDuality}
Let $C \in \mathcal{P}_{\mathrm{fin}}(\Z^d)$ and $E,F \in \mathcal{P}_{\mathrm{fin}}(E(\Z^d)) \text{ with } E \cap F = \varnothing$. Then, for any~$\mu \in \mathcal I$ we have that~$\Bar{\mu}( C,E,F) = \mathbf{E}_{C,E,F} \big[ \Bar{\mu}( \mathcal C_t, \mathcal A_t, \mathcal B_t ) \big]$ for all $t\ge 0$.
\end{corollary}
\begin{proof}
We have for any $ t\ge0$,
\begin{equation*}
    \begin{aligned}
        \Bar{\mu}( C,E,F) & = p^{-|E|}(1-p)^{-|F|} \cdot \widehat{\mu}( C,E,F) \\[.2cm]
        & = p^{-|E|}(1-p)^{-|F|} \cdot \int P_{\eta,\zeta}(\eta_t \equiv 1 \text{ on } C,\; \zeta_t \equiv 1 \text{ on } E,\; \zeta_t \equiv 0 \text{ on } F ) \;\mu( \mathrm{d}(\eta,\zeta) ) \\[.2cm]
        & \hspace{-1.5mm} \stackrel{(\ref{DualityNice})}{=} \int \mathbf{E}_{C,E,F} \big[ p^{-|\mathcal A_t|}(1-p)^{-|\mathcal B_t|} \cdot \phi(\mathcal C_t,\eta) \cdot \varphi(\mathcal A_t,\mathcal B_t,\zeta) \big] \;\mu( \mathrm{d}(\eta,\zeta) ) \\[.2cm]
        & \stackrel{(\ast)}{=} \mathbf{E}_{C,E,F} \big[ \Bar{\mu}( \mathcal C_t, \mathcal A_t, \mathcal B_t ) \big],  
    \end{aligned}
\end{equation*}
where $(\ast)$ holds by Fubini's Theorem.
\end{proof}

\subsection{Characterization of stationary measures}\label{section_characterization_stat_measures}
We start this section by constructing a coupling of a system of coalescing random walks on dynamical percolation with a system of independent random walks on a single dynamical percolation.

Fix $R,d,k \in \N$, $p \in [0,1]$, $\mathsf{v} > 0$. Given a graphical construction $\mathcal{U}$ for dynamical percolation with parameters $p$ and $\mathsf{v}$, points $x_1,\dots,x_k$, $\zeta \in \Omega_\mathrm{edge}$ and instruction manuals~$\mathcal T_1, \dots, \mathcal T_k$ for random walks with range $R$ we define the process
$$ \big( \mathbb{X}_t^{\mathcal U, \mathcal{T}_1}(\zeta,x_1), \dots, \mathbb{X}_t^{\mathcal U, \mathcal{T}_k}(\zeta,x_k), \mathbb{A}_t^{\mathcal U, \mathbf{T}}(\zeta,\mathbf{x}), \mathbb{B}_t^{\mathcal U, \mathbf{T}}(\zeta,\mathbf{x}) \big), \; t \geq 0. $$
We pointed out in Remark~\ref{remark_mc_xab} that this process is a Markov chain in the case when $\zeta_0 \sim \pi^\mathrm{edge}_p$ and the process is started from $(x_1,\dots,x_k,\varnothing,\varnothing)$, i.e., there is no initial knowledge of the environment. To obtain the desired coupling we need to construct this Markov chain so that it allows some initial knowledge. To do so, we define the following objects: let $E,F \in \mathcal{P}_\mathrm{fin}(E(\Z^d))$ with $E \cap F = \varnothing$, 
\begin{equation*}
    \begin{aligned}
        & E_t := \{e \in E: \mathcal{U}^e \cap [0,t] \neq \varnothing \}, \; t \ge 0;\\
        & F_t := \{e \in F: \mathcal{U}^e \cap [0,t] \neq \varnothing \}, \; t \ge 0;\\
        & \bar \zeta \in \Omega_\mathrm{edge}: \bar \zeta \equiv 1 \text{ on } E \text{ and } \bar \zeta \equiv 0 \text{ on } E^c.
    \end{aligned}
\end{equation*}
Let $\zeta_0 \sim \pi^\mathrm{edge}_p$, we define the configuration $\zeta_0^{\bar \zeta \to E \cup F}$. See~\eqref{XiEzeta}.  Then it follows that
\begin{equation}\label{MC_rws_starting_xyEF}
    \big( \mathbb X_t^{\mathcal U, \mathcal T_1}(\zeta_0^{\bar \zeta \to E \cup F},x_1),\dots, \mathbb X_t^{\mathcal U, \mathcal T_k}(\zeta_0^{\bar \zeta \to E \cup F},x_k),\; \mathbb{A}_t^{\mathcal U, \mathbf{T}}(\zeta_0^{\bar \zeta \to E \cup F}, \mathbf{x}) \cup E_t,\; \mathbb{B}_t^{\mathcal U, \mathbf{T}}(\zeta_0^{\bar \zeta \to E \cup F}, \mathbf{x}) \cup F_t \big)
\end{equation}
is a Markov chain started from $(x_1,\dots,x_k,E,F)$. Recall the notation $\mathbf x^{i,y}$ from~\eqref{x_bf_i_y}. The jumping rates of the Markov chain~\eqref{MC_rws_starting_xyEF} are
\begin{equation}\label{matrix_of_rates_xoc_process}
    \begin{array}{lll}
    q\big((\textbf{x},E,F);(\textbf{x},E \setminus \{e\}, F\setminus \{e\})\big) & = \mathsf{v}, & e \in E \cup F; \\[.2cm]
    q\big((\textbf{x},E,F);(\textbf{x}^{i,y},E \cup E', F \cup F')\big) & = \frac{p^{|E'|}(1-p)^{|F'|}}{|B_1(R)|-1}, & i=1,\dots,k, y \in B_1(x_i,R) \setminus \{x_i\},\\
    & & (E',F') \in \mathcal{B}^R_{x_i \to y}(E,F);  \\[.2cm]
    q\big((\textbf{x},E,F);(\textbf{x},E \cup E', F \cup F')\big) & = \frac{p^{|E'|}(1-p)^{|F'|}}{|B_1(R)|-1}, & i=1,\dots,k, y \in B_1(x_i,R) \setminus \{x_i\}, \\ 
    & & (E',F') \in \mathcal{B}^R_{x_i \not \to y}(E,F). 
\end{array}
\end{equation}

Now, to construct the coupling with a system of coalescing random walks. To abbreviate notation, we write
$$ (X^i_t,\mathbb{A}^i_t, \mathbb{B}^i_t) = ( \mathbb{X}_t^{\mathcal U, \mathcal{T}_i}(\zeta_0^{\bar \zeta \to E \cup F},x_i), \mathbb{A}_t^{\mathcal U, \mathcal T_i}(\zeta_0^{\bar \zeta \to E \cup F},x_i), \mathbb{B}_t^{\mathcal U, \mathcal T_i}(\zeta_0^{\bar \zeta \to E \cup F},x_i) ),\;  i = 1,\dots,k. $$
We define the processes $(Y^{x_i}_t, A^{x_i}_t, B^{x_i}_t)_{t \geq 0}$ for $i=1,\dots,k$ by induction, as follows. Put $(Y^{x_1}_t,A^{x_1}_t, B^{x_1}_t) = (X^{x_1}_t, \mathbb{A}^{x_1}_t, \mathbb{B}^{x_1}_t)$ for all $t \geq 0$. Assume~$(Y^{x_1}_t, A^{x_1}_t, B^{x_1}_t)_{t \ge 0}, \dots, (Y^{x_n}_t,A^{x_n}_t, B^{x_n}_t)_{t \ge 0}$ are defined and let
$$ \sigma = \inf \{ t \geq 0: X^{x_{n+1}}_t = Y^{x_i}_t \text{ for some } i \leq n \}. $$
\begin{itemize}
    \item On $\{\sigma = \infty\}$, let $(Y^{x_{n+1}}_t,A^{x_{n+1}}_t,B^{x_{n+1}}_t) = (X^{x_{n+1}}_t, \mathbb{A}^{x_{n+1}}_t, \mathbb{B}^{x_{n+1}}_t)$ for all $t \geq 0$. 
    \item On $\{\sigma < \infty\}$, let $K$ be the smallest index such that $X^{x_{n+1}}_\sigma = Y^{x_K}_\sigma$. Put
$$ (Y^{x_{n+1}}_t, A^{x_{n+1}}_t, B^{x_{n+1}}_t) = \left\{ \begin{array}{ll}
    (X^{x_{n+1}}_t,\mathbb{A}^{x_{n+1}}_t,\mathbb{B}^{x_{n+1}}_t) & \text{if } t \leq \sigma  \\[.2cm]
    (Y^{x_K}_t,\mathbb{A}^{x_K}_t \cup E^{x_{n+1}}_t, \mathbb{B}^{x_K}_t \cup F^{x_{n+1}}_t) & \text{if } t > \sigma, 
\end{array} \right. $$
where $ E^{x_{n+1}}_t :=\{ e \in E^{x_{n+1}}_\sigma: \mathcal{U}^e \cap [\sigma,t] = \varnothing \}$ and $ F^{x_{n+1}}_t :=\{ e \in F^{x_{n+1}}_\sigma: \mathcal{U}^e \cap [\sigma,t] = \varnothing \}$. 
\end{itemize}
We see that
$$(C_t,A_t,B_t) = \big( \big\{ Y^{x_i}_t: i=1,\dots,k \big\}, \big( \cup^k_{i = 1} A^{x_i}_t \big) \cup E_t, \big( \cup^k_{i = 1} B^{x_i}_t \big) \cup F_t\big) $$
is a Markov chain with jumping rates \eqref{coalescing_jump_rate} started from $(\{x_1,\dots,x_k\},E,F)$ with the following properties. First, let $ \mathcal X_t := \{ X^i_t: i = 1,\dots,k \}$, then $(C_t, A_t, B_t)$ and $(\X_t,\mathbb{A}_t,\mathbb{B}_t)$ agree until the first time $t$ such that $|\X_t|<|\X_0|$. Second, $C_t \subseteq \X_t$, $A_t \subseteq \mathbb{A}_t$ and~$B_t \subseteq \mathbb{B}_t$ for all~$t \geq 0$. Then, for any~$|f| \leq 1$,
\begin{equation}\label{coupling_AET_XET}
    \big| \E[f(C_t, A_t, B_t)] - \E[f(\X_t,\mathbb{A}_t,\mathbb{B}_t)] \big| \leq \Prob(\exists t \geq 0: |\X_t|<|\X_0|). 
\end{equation}
Now we use \eqref{coupling_AET_XET} to prove Theorem \ref{ThmCharacterization}. 

\subsubsection{Dimension 1 and 2}
Let $\mu \in \mathcal{I}$. Then, for every $t \geq 0$, $x,y \in \Z^d$, writing $\mathcal{C}_t = \{Y^1_t,Y^2_t\}$, we have by Corollary~\ref{CoroDuality} that
\begin{align*}
    \mu\big( (\eta,\zeta): \eta(x) \neq \eta(y) \big) & = \mathbf{E}_{ \{x,y\},\varnothing,\varnothing} \big[ \mu( (\eta,\zeta): \eta(Y^1_t) \neq \eta(Y^2_t) \mid \zeta \equiv 1 \text{ on } \mathcal A_t, \zeta \equiv 0 \text{ on } \mathcal B_t) \big]\\
    & \leq  \mathbf{P}_{ \{x,y\},\varnothing,\varnothing}( Y^1_t \neq Y^2_t ) \stackrel{(\ast)}= \Prob( \forall s \leq t: X^1_s \neq X^2_s \mid \zeta_0 \sim \pi^\mathrm{edge}_p),    
\end{align*}
where $(\ast)$ holds by the coupling with independent random walks on dynamical percolation and $(X^1_t,X^2_t,\zeta_t)$ is a system of two independent random walks on a single percolation environment. We take the limit in the last equation as~$t \to \infty$ and by Proposition \ref{CollisionRWd12} we conclude that~$\mu$ is supported on~$\{ \overline{0},\overline{1} \} \times \Omega_\mathrm{edge}$. It is easy to see that the edge marginal of the stationary measure~$\mu$ is $\pi^\mathrm{edge}_p$, so $\mu = \alpha \delta_{ \bar{1} } \otimes \pi^\mathrm{edge}_p + (1-\alpha) \delta_{ \bar{0}} \otimes \pi^\mathrm{edge}_p$ for some $\alpha \in [0,1]$. This proves that~$\mathcal{I}_e = \{ \delta_{ \bar{1} } \otimes \pi^\mathrm{edge}_p,\delta_{ \bar{0} } \otimes \pi^\mathrm{edge}_p \}$.

\subsubsection{Dimension 3 and higher}
Fix~$\mu \in \mathcal{I}$ for the rest of this section. We aim to show that there is a measure~$\gamma$ on~[0,1] such that
$$ \mu = \int^1_0 \mu^\mathrm{dyn}_\alpha \; \gamma( \mathrm{d} \alpha).$$
Let~$C \in \mathcal{P}_\mathrm{fin}(\Z^d)$,~$E,F \in \mathcal{P}_\mathrm{fin}(E(\Z^d))$ with~$E \cap F = \varnothing$. Then, by the coupling equation~\eqref{coupling_AET_XET}, 
\begin{equation}\label{Ineq1}
    \begin{aligned}
    \big |\mathbf{E}_{C,E,F} \big[ \Bar{\mu}(\mathcal C_t, \mathcal A_t, \mathcal B_t) \big] - \mathbf{E}_{C,E,F} \big[ \Bar{\mu}(\X_t,\mathbb{A}_t,\mathbb{B}_t) \big] \big| \leq \mathbf{P}_{C,E,F}( \exists t \geq 0: |\X_t| < |\X_0|).    
    \end{aligned}
\end{equation}
We define
$$ g_\mathrm{dyn}(C,E,F) := \mathbf{P}_{C,E,F}( \exists t\geq 0: |\X_t| < |\X_0| ) .  $$
Recalling Corollary~\ref{CoroDuality}, we have from (\ref{Ineq1}) that
\begin{equation}\label{Equationmu}
    \big| \Bar{\mu}(C,E,F) -  \mathbf{E}_{C,E,F} \big[ \Bar{\mu}(\X_t,\mathbb{A}_t,\mathbb{B}_t) \big] \big| \leq g_\mathrm{dyn}(C,E,F).
\end{equation}
Now, our goal is to study $\mathbf{E}_{C,E,F} \big[ \Bar{\mu}(\X_t,\mathbb{A}_t,\mathbb{B}_t) \big]$ as $t \to \infty$.
\begin{lemma}\label{QProperty}
    $\mathbf{E}_{C,E,F} \big[ \Bar{\mu}(\X_t,\mathbb{A}_t,\mathbb{B}_t) \big]$ converges as $t \to \infty$.
\end{lemma}
\begin{proof}
For any $s,t \geq 0$,
\begin{equation}\label{inequality_bar_mu_chain}
    \begin{aligned}
        \big| \mathbf{E}_{C,E,F} & \big[ \Bar{\mu}(\X_t,\mathbb{A}_t,\mathbb{B}_t) \big] - \mathbf{E}_{C,E,F} \big[ \Bar{\mu}(\X_{t+s},\mathbb{A}_{t+s},\mathbb{B}_{t+s}) \big] \big| \\[.2cm]
        & \stackrel{(\ast)}{=} \Big| \mathbf{E}_{C,E,F} \Big[ \mathbf{E}_{\X_t,\mathbb{A}_t,\mathbb{B}_t}  \big[ \Bar{\mu}(\mathcal C_s,\mathcal A_s,\mathcal B_s) \big] \Big] - \mathbf{E}_{C,E,F} \Big[ \mathbf{E}_{\X_t,\mathbb{A}_t,\mathbb{B}_t} \big[ \Bar{\mu}(\X_s,\mathbb{A}_t,\mathbb{B}_t) \big] \Big] \Big| \\[.2cm]
        & \hspace{-1.5mm} \stackrel{(\ref{Ineq1})}{\leq} \mathbf{E}_{C,E,F} \big[ g(\X_t, \mathbb{A}_t,\mathbb{B}_t) \big], \\
    \end{aligned}
\end{equation}
where $(\ast)$ holds because of Corollary~\ref{CoroDuality} and the Markov property. Convergence follows from proving that
\begin{equation}\label{EgXt}
    \lim_{t \to \infty} \mathbf{E}_{C,E,F}\big[ g_\mathrm{dyn}(\X_t,\mathbb{A}_t,\mathbb{B}_t) \big] = 0.
\end{equation}
By the Markov property, we have that
$$ \mathbf{E}_{C,E,F}\big[ g_\mathrm{dyn}(\X_t,\mathbb{A}_t,\mathbb{B}_t) \big] = \mathbf{P}_{C,E,F}( \exists s \ge t: |\mathcal X_s| < |\mathcal X_t| ).$$
Using Corollary \ref{CoroBigThm} with $\ell = 0$ and a union bound, we have the statement \eqref{EgXt}.
\end{proof}
In view of the above lemma, we define
\begin{equation}\label{q_def_E_CEF}
    q(C,E,F):= \lim_{t \to \infty} \mathbf{E}_{C,E,F} \big[ \Bar{\mu}(\X_t,\mathbb{A}_t,\mathbb{B}_t) \big].
\end{equation}

\begin{lemma}\label{q_CEF_constant}
Given $k \in \N$, there exists a constant $c = c(k)$ such that
$$q(C,E,F) = c(k) \quad \forall \; C \in \mathcal{P}_\mathrm{fin}(\Z^d) \text{ with } |C|=k, \; E,F \in \mathcal{P}_\mathrm{fin}(E(\Z^d)) \text{ with } E\cap F = \varnothing.$$
\end{lemma}
\begin{proof}
We present the proof for $k=2$ but it works for any $k \in \N$ with minor modifications. We first consider the case $E = F = \varnothing$. We aim to use the coupling of Lemma~\ref{couple_single_separate}.

Let $x,y \in \Z^d$. By the Markov property, it follows that~$q(\{X^1_t, X^2_t\}, \mathbb{A}_t, \mathbb{B}_t), \; t \ge 0,$ is a bounded martingale. The martingale convergence theorem implies that 
\begin{equation}\label{convergence_as_q}
    q( \{X^1_t, X^2_t\} , \mathbb{A}_t, \mathbb{B}_t) \text{ converges } \mathbf{P}_{x,y,\varnothing,\varnothing}\text{-a.s. as } t \to \infty.  
\end{equation}
At this point, we want to argue that the limit is almost surely constant. The way we did this when treating the classical voter model was by appealing to a 0-1 law. The slight issue we have here is that, since~$(X^1_t)$ and~$(X^2_t)$ share the same environment, they do not evolve independently, and there is no clear way to use regeneration times to argue that independent random walks can be embedded inside these processes. We now work around this obstacle.

Let $(X_t, A^1_t, B^2_t, Y_t, A^2_t, B^2_t)_{t \ge 0}$ be the Markov chain~\eqref{MC_separate_RW}. We want to prove that the process~$q \big( \{ X_t, Y_t \}, A^1_t \cup A^2_t, B^1_t \cup B^2_t \big)$ converges a.s. as $t \to \infty$. From Lemma~\ref{couple_single_separate}, there exists a coupling $(\mathbf{X}_t,\mathbf{Y}_t)$ under~$\widehat{\Prob}$ such that $\mathbf{X}_t$ and $\mathbf{Y}_t$ have jumping rate $r_\mathrm{single}$ and $r_\mathrm{separate}$, respectively, and
\begin{equation}\label{coupling_ineq_1_XY_t}
    \widehat{\Prob}( \exists t \ge 0: \mathbf{X}_t \neq \mathbf{Y}_t ) \leq 2 g_{2R}(x,y).
\end{equation}
Recall the regeneration times $\{\sigma_n\}_{n \in \N}$ from~\eqref{reg_times_def}. We have that
\begin{equation}\label{h_def_1_2g}
    \begin{aligned}
        h(x,y) & := \Prob \big( q \big( \{ X_{\sigma_n}, Y_{\sigma_n} \}, \varnothing, \varnothing \big) \text{ converges as } n \to \infty \big) \\[0.2cm]
        & \geq \Prob \big( q \big( \{ X_t, Y_t \}, A^1_t \cup A^2_t, B^1_t \cup B^2_t \big) \text{ converges as } t \to \infty \big) \\[0.2cm]
        & \stackrel{(\ast)} \geq \widehat{\Prob}( \forall \; t \ge 0, \; \mathbf{X}_t = \mathbf{Y}_t ) \stackrel{\eqref{coupling_ineq_1_XY_t}}\ge 1 - 2g_{2R}(x,y),
    \end{aligned}
\end{equation}      
where $(\ast)$ holds by noticing that if~\eqref{convergence_as_q} holds, then evaluating $q$ at the corresponding process constructed from $(\mathbf{X}_t)$ will also converges a.s. as $t \to \infty$. Given $\epsilon>0$, by Lemma~\ref{lemma_limit_fl_g}, there exists $L>0$ such that for~$|x-y|_1 \ge L$,~$g_{2R}(x,y)<\epsilon/2$. Assume now that $|x-y| < L$. Recall that the random walk~$X_{\sigma_n}-Y_{\sigma_n}$ has increments with distribution~$U$ as in~\eqref{increments_rw_U}, then it is transient (See Theorem~4.1.1. of~\cite{Lawler2010}). Hence, there exists $N \in \N$ such that
\begin{equation}\label{transient_L_distance}
    \Prob ( |X_{\sigma_N}-Y_{\sigma_N}|_1 \ge L ) > 1 - \epsilon.
\end{equation}
By the Strong Markov property 
\begin{align*}
   h(x,y) = \E [ h(X_{\sigma_N}, Y_{\sigma_N}) ] & \ge \E [ h(X_{\sigma_N}, Y_{\sigma_N}) \cdot \mathds{1}\{ |X_{\sigma_N}-Y_{\sigma_N}|_1 \ge L \} ] \\
   & \hspace{-2.4mm} \stackrel{\eqref{h_def_1_2g}} \ge (1-\epsilon) \cdot \Prob ( |X_{\sigma_N}-Y_{\sigma_N}|_1 \ge L ) \stackrel{\eqref{transient_L_distance}}> 1- 2\epsilon. 
\end{align*}
Therefore, the sequence~$q \big( \{ X_{\sigma_n},Y_{\sigma_n} \}, \varnothing, \varnothing\big)$ converges a.s. as $n \to \infty$ for any initial positions. The Hewitt-Savage 0-1 law implies that for any $x,y \in \Z^d$, 
$$ q(\{x,y\}, \varnothing, \varnothing) = \lim_{n \to \infty} q \big( \{ X_{\sigma_n},Y_{\sigma_n} \}, \varnothing, \varnothing\big) \; \text{ a.s.}$$
This proves that $q$ is harmonic for the irreducible random walk~$\big(X_{\sigma_n},Y_{\sigma_n}\big)$ on~$(\Z^d)^2$, we conclude that there exists a constant $c = c(2)$ such that~$q(\{x,y\}, \varnothing, \varnothing) = c$ for any~$x,y \in \Z^d$.

Turning now to the general case, we want to prove that~$q(\{x,y\},E,F) = c(2)$ for any~$E,F \in \mathcal{P}_\mathrm{fin}(E(\Z^d))$ with $E \cap F = \varnothing$. To do so, we extend the definition of regeneration times (Section \ref{SectionRegenrationTimes}) for two independent random walks on a single dynamical percolation environment. Let $(X^1_t,X^2_t,\mathbb{A}_t, \mathbb{B}_t)_{t \ge 0}$ be the Markov chain with jumping rate~\eqref{matrix_of_rates_xoc_process} with~$k=2$. We define the stopping times~$\{\tau_n\}_{n \in \N_0}$ (cf.~\eqref{reg_times_def}), by setting 
$$\tau_0 = \inf\{ s \geq 0: \mathbb{A}_s = \mathbb{B}_s = \varnothing \},$$
and for $n \ge 1$,
$$ \tau_n = \inf\{ m \in \N : m > \tau_{n-1} \text{ and } \mathbb{A}_m \cup \mathbb{B}_m = \varnothing \}. $$
The same standard arguments to obtain Properties 1 and 2 for the regeneration times~$\sigma_n$~$\eqref{reg_times_def}$ show that $\tau_n < \infty$ a.s. for all $n$. Then, for any $m \in \N$,
\begin{equation}\label{q_constant_k_general}
\begin{aligned}
    q(\{x,y\} ,E,F) & \stackrel{\eqref{q_def_E_CEF}}= \lim_{t \to \infty} \mathbf{E}_{x,y,E,F} \Big[ \bar{\mu} \big( \{X^1_{\tau_m + t}, X^2_{\tau_m + t}\},\mathbb{A}_{\tau_m + t},\mathbb{B}_{\tau_m + t} \big) \Big]  \\
    & \hspace{1.8mm} \stackrel{(\ast)} = \mathbf{E}_{x,y,E,F} \Big[ q \big(\{X^1_{\tau_m}, X^2_{\tau_m}\}, \varnothing, \varnothing \big) \Big] = \mathbf{E}_{x,y,E,F} \Big[c \big( \big| \{X^1_{\tau_m}, X^2_{\tau_m}\} \big| \big) \Big],
\end{aligned}
\end{equation}
where $(\ast)$ holds by the Strong Markov property. The transience of the random walks implies that taking the limit on~\eqref{q_constant_k_general} as $m \to \infty$, we have that~$q(\{x,y\},E,F) = c(2)$.
\end{proof} 
As a consequence of this lemma, we take the limit on (\ref{Equationmu}) as $t \to \infty$ to obtain that
\begin{equation}\label{FinettiIn}
    \big|\Bar{\mu}(C,E,F) - c(|C|) \big| \leq g_\mathrm{dyn}(C,E,F).
\end{equation}
Given $\mu \in \mathcal I$ and $E,F \in \mathcal P_\mathrm{fin}(E(\Z^d))$ with $E \cap F = \varnothing$, we consider the random configuration
$$ (\eta_{EF},\zeta_{EF}) \sim \mu( \hspace{1mm} \cdot \mid \{(\eta,\zeta): \zeta \equiv 1 \text{ on } E, \zeta \equiv 0 \text{ on } F \}).$$
\begin{lemma}
    There exists a probability measure $\gamma$ on $[0,1]$ such that
    $$ \mu = \int^1_0 \mu^\mathrm{dyn}_\alpha \; \gamma( \mathrm{d}\alpha ). $$
\end{lemma}
\begin{proof}
We enlarge the probability space $\Prob$ so that $\{(\eta_{EF},\zeta_{EF}): E,F \in \mathcal P_\mathrm{fin}(E(\Z^d)) \text{ with } E \cap F = \varnothing\}$ are also defined and are independent of the graphical construction $\mathcal U$ and the instruction manuals~$\mathcal T_i$. We set $\zeta_0 \sim \pi^\mathrm{edge}_p$. Let $\Xi^{EF}_t(x) = \eta_{EF}(\mathbb{X}_t(\zeta_0,x))$ for every $x \in \Z^d$ and $t \ge 0$, we define the measure~$\nu^\mathrm{dyn}_t$ on~$\Omega_\mathrm{site}$ by 
$$ \nu^\mathrm{dyn}_t( \cdot ) = \E[ \Prob( \Xi^{\mathbb A_t, \mathbb B_t}_t \in \cdot ) | \mathbb{A_t},\mathbb{B_t} ].$$
For any $k \in \N$, we have that
$$ \nu^\mathrm{dyn}_t \big(\eta: \eta(x_1) = \cdots = \eta(x_k) = 1 \big) = \E \big[ \bar{\mu} \big(\mathbb X_t(\zeta_0,\mathbf{x}),\mathbb A_t(\zeta_0,\mathbf{x}),\mathbb{B}_t(\zeta_0,\mathbf{x}) \big) \big].$$
Lemma~\ref{QProperty} implies that $\nu^\mathrm{dyn}_t$ converges weakly to a measure~$\nu^\mathrm{dyn}_\infty$ as $t \to \infty$. Hence,
\begin{equation}\label{nu_dyn_def_limit}
    \nu^\mathrm{dyn}_\infty \big(\eta: \eta(x_1) = \cdots = \eta(x_k) = 1 \big) = \lim_{t \to \infty}\E \big[ \bar{\mu}\big(\mathbb X_t(\zeta_0,\mathbf{x}),\mathbb A_t(\zeta_0,\mathbf{x}),\mathbb{B}_t(\zeta_0,\mathbf{x}) \big) \big] = c(k).
\end{equation}
From Lemma~\ref{q_CEF_constant}, we deduce that $\nu^\mathrm{dyn}_\infty$ is exchangeable, then, by de~Finetti's theorem there exists a probability measure $\gamma$ on $[0,1]$ so that
\begin{equation}\label{QuickEq1DeFinetti}
    \nu^\mathrm{dyn}_\infty = \int^1_0 \pi^\mathrm{edge}_\alpha \; \gamma(d\alpha).
\end{equation}
We define
\begin{equation}\label{mu_dyn_ast}
    \mu_\ast = \int^1_0 \mu^{\mathrm{dyn}}_\alpha \; \gamma(d \alpha),
\end{equation}
To complete the proof of the lemma we need to prove that $\mu = \mu_\ast$. We have that
$$\big|\Bar{\mu}^\mathrm{dyn}_\alpha(C,E,F) - \alpha^{|A|} \big| \stackrel{(\ref{mu_alpha_def_with_A_infty})}{=} \mathbf{E}_{C,E,F} \big[ \big| \alpha^{|\mathcal C_\infty|} - \alpha^{|A|} \big| \big] \leq g_\mathrm{dyn}(C,E,F).$$
Integrating over $\gamma$ and using \eqref{nu_dyn_def_limit}, \eqref{QuickEq1DeFinetti} and \eqref{mu_dyn_ast} we have that
$$\big|\Bar{\mu}_*(C,E,F) - c(|C|) \big| \leq g_\mathrm{dyn}(C,E,F).$$
Comparing this with (\ref{FinettiIn}), it follows that
\begin{equation}\label{FinalFinetti}
    \big| \Bar{\mu}_*(C,E,F) - \Bar{\mu}(C,E,F) \big| \leq 2 g_\mathrm{dyn}(C,E,F).
\end{equation}
Since $\mu_*$ is also stationary for the voter model on dynamical percolation, by Corollary~\ref{CoroDuality}
\begin{equation}\label{Eq2}
    \begin{aligned}
    | \Bar{\mu}_*(C,E,F) - \Bar{\mu}(C,E,F) | &  = \big| \mathbf{E}_{C,E,F} \big[ \Bar{\mu}_*(\mathcal C_t, \mathcal A_t, \mathcal B_t)\big] - \mathbf{E}_{C,E,F} \big[ \Bar{\mu}(\mathcal C_t, \mathcal A_t, \mathcal B_t) \big] \big| \\
    & \hspace{-2.2 mm} \stackrel{(\ref{FinalFinetti})}{\leq} 2 \mathbf{E}_{C,E,F} \big[ g_\mathrm{dyn}(\mathcal C_t, \mathcal A_t, \mathcal B_t) \big] \stackrel{(\ast)} \leq 2 \E \big[ g_\mathrm{dyn}(\X_t, \mathbb{A}_t,\mathbb B_t) \big],
    \end{aligned}
\end{equation}
where $(\ast)$ holds by the coupling described at the beginning of this section. It follows from~\eqref{EgXt} that $\mu = \mu_*$. 
\end{proof}
By the Krein-Milman theorem
$$ \mathcal{I}_e \subseteq \mathrm{cl}(\{ \mu^{\mathrm{dyn}}_\alpha: \alpha \in [0,1] \}).$$
It can be checked that the family $\{ \mu^{\mathrm{dyn}}_\alpha: \alpha \in [0,1] \}$ is closed. Conversely, the fact that~$\mu_\alpha \in \mathcal{I}_e$ for every $\alpha$ follows from the same method described in the proof of Theorem \ref{MetaTheorem}.

\subsection{Ergodicity of the stationary measures}\label{section_ergodicty_dyn_perc}
We start introducing some notation. For~$x \in \Z^d, e =\{x_1,x_2\} \in E(\Z^d)$, we define the shift transformation~$\tau_x$ on~$\Omega_\mathrm{site} \times \Omega_\mathrm{edge}$ by
$$ \big(\tau_x (\eta,\zeta)\big)(y,e) = (\eta,\zeta)(y-x,e-x),$$
where~$e-x = \{x_1-x,x_2-x\}$. This shift map induces the shift translation of subsets of~$\Omega_\mathrm{site} \times \Omega_\mathrm{edge}$. In particular, for $A \in \mathcal{P}_\mathrm{fin}(\Z^d)$ and $E \in \mathcal{P}_\mathrm{fin}(E(\Z^d))$,
\begin{equation*}
    \tau_x(\{ (\eta,\zeta): \eta \equiv 1 \text{ on } A, \zeta \equiv 1 \text{ on } E \}) = \{ (\eta,\zeta): \eta \equiv 1 \text{ on } A+x, \zeta \equiv 1 \text{ on } E+x \},
\end{equation*}
where $A + x = \{y + x: y \in A\}$ and $E+x = \{ e + x: x \in E \}$. 

For dimensions 1 and 2, spatial ergodicity follows immediately since for any $\alpha \in [0,1]$, the measure~$\mu^\mathrm{dyn}_\alpha$ is a convex combination of the consensus measures. Now we focus on dimension 3 and higher. For every~$\alpha \in [0, 1]$, the measures $\mu^{\mathrm{dyn}}_\alpha$ are easily seen to be spatial invariant, spatial ergodicity will follow from the mixing property: for any $A,B \in \mathcal{P}_\mathrm{fin}(\Z^d)$ and $E_1, E_2 \in \mathcal{P}_\mathrm{fin}( E(\Z^d))$, 
\begin{equation}\label{MixingProperty}
    \lim_{x \to \infty}\widehat{\mu}^{\mathrm{dyn}}_\alpha( A \cup (B + x), E_1 \cup (E_2 + x), \varnothing ) = \widehat{\mu}^\mathrm{dyn}_\alpha (A,E_1,\varnothing \big) \cdot \widehat{\mu}^\mathrm{dyn}_\alpha( B, E_2, \varnothing).
\end{equation}
We couple the dual processes~$(\mathcal C^1_t, \mathcal A^1_t, \mathcal B^1_t), (\mathcal C^2_t, \mathcal A^2_t,\mathcal B^2_t)$ and $(\mathcal C^3_t, \mathcal A^3_t,\mathcal B^3_t)$ in the probability space $\mathbf{P}$ with the same law given by the jump rates \eqref{coalescing_jump_rate} such that the following holds: 
\begin{itemize}
    \item $(\mathcal C^1_0, \mathcal A^1_0, \mathcal B^1_0) = ( A,E_1,\varnothing)$, $(\mathcal C^2_0, \mathcal A^2_0, \mathcal B^2_0) = (B,E_2,\varnothing)$, $(\mathcal C^3_0, \mathcal A^3_0, \mathcal B^3_0) = (A \cup B, E_1 \cup E_2,\varnothing)$;
    \item $(\mathcal C^i_t, \mathcal A^i_t, \mathcal B^i_t)$, $i=1,2$ are independent;
    \item $(\mathcal C^3_t, \mathcal A^3_t, \mathcal B^3_t)$ is defined as follows: let
$$\rho := \inf \{ t > 0: \min \big(\mathrm{dist}( \mathcal C^1_t, \mathcal A^2_t \cup \mathcal B^2_t ),\hspace{0.5mm} \mathrm{dist}( \mathcal C^2_t, \mathcal A^1_t \cup \mathcal B^1_t ),\hspace{0.5mm} \mathrm{dist}(\mathcal C^1_t, \mathcal C^2_t) \big) \leq R \},$$ 
\begin{enumerate}
    \item for $t \leq \rho$, $ \mathcal C^3_t = \mathcal C^1_t \cup \mathcal C^2_t, \hspace{1mm} \mathcal A^3_t = \mathcal A^1_t \cup \mathcal A^2_t, \hspace{1mm} \mathcal B^3_t = \mathcal B^1_t \cup \mathcal B^2_t$,
    \item for $t > \rho$, $(\mathcal C^3_t,\mathcal A^3_t,\mathcal B^3_t)$ moves independently of $(\mathcal C^i_t,\mathcal A^i_t,\mathcal B^i_t)$, $i=1,2$.
\end{enumerate}
\end{itemize}
Then
\begin{equation}\label{ErgodicityEquation}
    \begin{aligned}
    \big| \widehat{\mu}^\mathrm{dyn}_\alpha(A \cup B, E_1 \cup E_2, \varnothing)  - \widehat{\mu}^\mathrm{dyn}_\alpha(A,E_1,\varnothing) \cdot \widehat{\mu}^\mathrm{dyn}_\alpha(B,E_2,\varnothing) \big| & = \textbf{E} \big[ \alpha^{|\mathcal C^3_\infty|} \big] - \textbf{E} \big[ \alpha^{ |\mathcal C^1_\infty| + |\mathcal C^2_\infty| } \big] \\
    & \leq \textbf{P}( \rho < \infty ).
    \end{aligned}
\end{equation}
By the construction, we note that the coupling breaks when one of the walkers of $(\mathcal C^i_t)$ is at a distance less than $R$ from $E_j$ or from $(\mathcal{A}^j_t \cup \mathcal B^j_t)$, $(i,j) \in \{(1,2);(2,1)\}$, or a walker from $(\mathcal C^1_t)$ and another from $(\mathcal C^2_t)$ are at a distance less than $R$. This motivates the definition of the following functions: let $\mathcal U$ be a graphical construction for dynamical percolation, and assume that the environment starts at stationarity, and~$\mathcal T_1, \mathcal T_2$ be instruction manuals for random walks with range $R$. We write for all $t \ge 0$,
$$ (X^1_t,X^2_t) = (\mathbb{X}^{\mathcal U, \mathcal T_1}_t(\zeta_0,x),\mathbb{X}^{\mathcal U, \mathcal T_2}_t(\zeta_0,x)) $$
and
$$ (F^1_t,F^2_t) = \big( \mathbb{A}^{\mathcal U, \mathcal T_1}_t(\zeta_0,x) \cup \mathbb{B}^{\mathcal U, \mathcal T_1}_t(\zeta_0,x), \hspace{1mm} \mathbb{A}^{\mathcal U, \mathcal T_2}_t(\zeta_0,y) \cup \mathbb{B}^{\mathcal U, \mathcal T_2}_t(\zeta_0,y) \big)$$
Then, for any $x \in \Z^d$ and $E \in \mathcal{P}_\mathrm{fin}(\Z^d)$, put 
$$\varphi(x,E) := \Prob( \exists t \geq 0: \mathrm{dist}(X^1_t,E) \leq R \text{ before all the edges in } E \text{ refresh} ).$$
and
$$ \tilde f_R(x,y) := \Prob(  \exists t\ge 0: \min \big( |X^1_t - X^2_t|_1, \hspace{0.5mm} \mathrm{dist}(X_t,F^2_t), \hspace{0.5mm} \mathrm{dist}(Y_t,F^1_t) \big) \leq R ).$$
Having in mind the moment when the coupling breaks we have the following inequality
\begin{equation}\label{ergo_result}
    \textbf{P}( \rho < \infty ) \leq p^{-|E_1|} \cdot \sum_{u \in A} \varphi(u,E_2) + p^{-|E_2|} \cdot \sum_{u \in B} \varphi(u,E_1) + p^{-|E_1 \cup E_2|} \sum_{ u \in A,v \in B } \tilde f_R(u,v).
\end{equation}
We prove that the three sums on the right-hand side converge to 0 as~$|x|_1 \to \infty$ when replacing~$B$ and~$E_2$ by~$B+x$ and~$E_2+x$, respectively. Let
\begin{align*}
    &\tau_E := \inf\{ t\geq 0: \text{each } e \in E \text{ refresh by time } t \},\\[.2cm]
    & L(\tau_E) := |\{ \text{attempted steps by a random walk by time } \tau_E \} |. 
\end{align*}
Then, 
\begin{equation}\label{ergo_result_1}
    \begin{aligned}
        \varphi(u,E) & \leq \Prob \big( R \cdot L(\tau_E) \geq \mathrm{dist}(u,E) \big) \leq \frac{R}{\mathrm{dist}(u,E)} \cdot \E[L(\tau_{E})] \\
        & = \frac{R}{\mathrm{dist}(u,E)} \cdot \E[\tau_{E}] \leq \frac{R}{\mathrm{dist}(u,E)} \cdot c(|E|), 
    \end{aligned}
\end{equation}
hence~$\varphi(u,E_2 + x)$ and~$\varphi(u+x,E_1)$ converge to 0 as~$|x|_1 \to \infty$. Convergence of~$\tilde f_R(u,v+x)$ to~0 as $|x|_1 \to \infty$ follows from Lemma~\ref{lemma_main} and Lemma~\ref{lemma_limit_fl_g}. Plugging this back in \eqref{ergo_result} and then in \eqref{ErgodicityEquation}, we obtain the mixing property.\\[-.2cm]

\textbf{Acknowledgements:} I am grateful to my advisor, Daniel Valesin, for suggesting the problem and for valuable discussions. I especially appreciate his careful reading of several versions of the manuscript, as well as his detailed comments and suggestions, which greatly contributed to the improvement of the paper. This work was carried out under the financial support of the CogniGron research center and the Ubbo Emmius Funds (Univ. of Groningen); the author is thankful for this support.

\bibliographystyle{plain}
\bibliography{Ref.bib}
\end{document}